\documentclass[leqno,11pt]{amsart}
\usepackage{amssymb, amsmath,amsmath,latexsym,amssymb,amsfonts,amsbsy, amsthm}
\usepackage[latin1]{inputenc}
\usepackage[english]{babel}
\usepackage{verbatim}
\usepackage{dsfont}
\usepackage{fancyhdr}
\usepackage{mathtools}
\usepackage{mathrsfs}
\usepackage{color}
\usepackage{graphicx}
\usepackage{bm}
\usepackage{enumerate}

\def\inte#1{
\displaystyle\mathop{#1\kern0pt}^\circ }

\def\tn#1{|\kern -0.08em\|{#1}|\kern -0.08em\|}

\def\eqdefa{\buildrel\hbox{\footnotesize def}\over =}

\newcommand{\beqo}{\begin{equation*}}
\newcommand{\eeqo}{\end{equation*}}
\newcommand{\beno}{\begin{eqnarray*}}
\newcommand{\eeno}{\end{eqnarray*}}

\numberwithin{equation}{section}

\newcommand{\mbb}{\mathbb}

\newcommand{\mc}{\mathcal}

\newcommand{\veps}{\varepsilon}
\newcommand{\what}{\widehat}
\newcommand{\wtilde}{\widetilde}
\newcommand{\vphi}{\varphi}
\newcommand{\oline}{\overline}

\newcommand{\ra}{\rightarrow}

\newcommand{\g}{\gamma}
\newcommand{\lan}{\langle}
\newcommand{\ran}{\rangle}

\newcommand{\R}{\mathbb{R}}
\newcommand{\Q}{\mathbb{Q}}

\newcommand{\T}{\mathbb{T}}
\newcommand{\N}{\mathbb{N}}
\renewcommand{\P}{\mathbb{P}}
\renewcommand{\Q}{\mathbb{Q}}
\newcommand{\A}{\mathbb{A}}
\newcommand{\Z}{\mathbb{Z}}
\renewcommand{\L}{\mathbb{L}}

\newcommand{\Id}{{\rm Id}\,}

\newcommand{\adj}{{\rm adj}\,}

\allowdisplaybreaks

\let\vrho=\varrho

\let\pa=\partial

\let\g=\gamma
\let\d=\delta

\let\s=\sigma
\let\f=\frac

\let\p=\psi

\let\D=\Delta

\let\wt=\widetilde
\let\wh=\widehat

\def\cA{{\mathcal A}}

\def\cC{{\mathcal C}}

\def\cF{{\mathcal F}}

\def\cL{{\mathcal L}}

\def\cN{{\mathcal N}}

\def\cS{{\mathcal S}}

\def\cX{{\mathcal X}}

\def\virgp{\raise 2pt\hbox{,}}
\def\cdotpv{\raise 2pt\hbox{;}}

\def\eqdefa{\buildrel\hbox{\footnotesize def}\over =}
\def\Id{\mathop{\rm Id}\nolimits}

\newcommand{\la}{\lambda}

\def\div{\mathop{\rm div}\nolimits}
\def\dive{\mathop{\rm div}\nolimits}

\def\na{\nabla}
\def\p{\partial}
\def\th{\theta}

\newcommand{\beq}{\begin{equation}}
\newcommand{\eeq}{\end{equation}}
\newcommand{\ben}{\begin{eqnarray}}
\newcommand{\een}{\end{eqnarray}}
\newcommand{\andf}{\quad\hbox{and}\quad}
\newcommand{\with}{\quad\hbox{with}\quad}

\newtheorem{defi}{Definition}[section]
\newtheorem{thm}{Theorem}[section]
\newtheorem{lem}{Lemma}[section]
\newtheorem{rmk}{Remark}[section]
\newtheorem{col}{Corollary}[section]
\newtheorem{prop}{Proposition}[section]

\title[Viscous compressible fluids with only bounded density]{A well-posedness result for viscous compressible fluids with only bounded density}

\author[R. Danchin]{Rapha\"el Danchin}
\address[R. Danchin]
{Universit\'e Paris-Est, UPEC, Laboratoire d'Analyse et de Math\'ematiques Appliqu\'ees, CNRS UMR 8050}
\email{raphael.danchin@u-pec.fr}

\author[F. Fanelli]{Francesco Fanelli}
\address[F. Fanelli]
{Univ. Lyon, Universit\'e Claude Bernard Lyon 1, CNRS UMR 5208, Institut Camille Jordan, 43 blvd. du 11 novembre 1918, F-69622 Villeurbanne cedex, France}
\email{fanelli@math.univ-lyon1.fr}

\author[M. Paicu]{Marius Paicu}
\address[M. Paicu]
{Universit\'e de Bordeaux, Institut de Math\'ematiques de Bordeaux, CNRS UMR 5251}
\email{marius.paicu@math.u-bordeaux.fr}

\date{\today}

\begin{document}

\begin{abstract}
We are concerned with the existence and uniqueness of solutions with only bounded 
density for the barotropic compressible Navier-Stokes equations.
 Assuming  that the initial velocity has slightly sub-critical regularity
and that the initial density is a small perturbation (in the $L^\infty$ norm) of a positive constant, 
we prove the existence of local-in-time solutions. 
In the case where the density takes two constant values across a smooth interface
(or, more generally, has  striated regularity with respect to 
some nondegenerate family of vector-fields), we get uniqueness. 
This latter result supplements the work by D. Hoff in \cite{Hoff2} with a uniqueness statement, 
and is valid in any dimension $d\geq2$ and for general pressure laws.
  \end{abstract}
\maketitle

\paragraph*{\bf 2010 Mathematics Subject Classification:}{\small 35Q35 
(primary);
35B65, 
76N10, 
35B30, 
35A02  
(secondary).}

\medbreak

\section*{Introduction}

We are concerned with  the multi-dimensional
barotropic compressible Navier-Stokes system in the whole space: \beq\label{NS}
\left\{
\begin{array}
{l} \displaystyle \pa_t\rho + \dive(\rho u) = 0\,,\\[1ex]
\displaystyle \pa_t (\rho u) + \dive(\rho u \otimes u) - \mu \Delta u
-\la \na \dive u + \nabla P(\rho) = 0\,.
\end{array}
\right. \eeq Here $\rho=\rho(t,x)$ and $u=u(t,x)$,  with $ (t,x)\in\R_+\times\R^d$
and $d\geq1$, 
denote the density and
velocity of the fluid, respectively.   The pressure $P$ is a given function
of~$\rho.$ We shall take that function locally in  $W^{1,\infty}$  
in all that follows, and assume (with no loss of generality) that it vanishes at some  constant 
reference density $\bar\rho>0.$
The (constant) viscosity coefficients $\mu$ and $\la$ satisfy \beq\label{lame} \mu>0\quad \mbox{and}\quad
\nu:=\la+\mu>0, \eeq which ensures ellipticity of  the operator
\begin{equation} \label{def:L}
\cL\,:=\,-\,\mu\, \Delta\, -\,\la\, \na \dive.
\end{equation}
We supplement System \eqref{NS} with the initial conditions
\begin{equation}\label{eq:initial}
\rho|_{t=0}=\rho_0\quad\hbox{and}\quad
u|_{t=0}=u_0.\end{equation}
\medbreak
A number of recent works have been dedicated 
to the study of solutions  with discontinuous density
(so-called ``shock-data'')  for  models of viscous compressible fluids.
 Even though the situation   is by now quite well understood
for $d=1$ (see e.g. \cite{Hoff0}),  
the multi-dimensional case is far from being completely elucidated. 
In this direction, we mention the papers \cite{Hoff1,Hoff2} by D. Hoff,
 who greatly contributed to construct   ``intermediate solutions'' allowing for discontinuity of the density,  in between the weak solutions of P.-L. Lions in \cite{PLL2} and 
the classical ones of e.g. J. Nash in \cite{Nash}.

In \cite{Hoff2}, in the two-dimensional case, D. Hoff succeeded to get quite an accurate information
on the propagation of density  discontinuities   accros  suitably smooth curves (as predicted by the Rankine-Hugoniot 
condition), under the assumption that the pressure is a linear function of $\rho$ (see Theorem 1.2 therein). 
 In particular, he proved  that those curves  are convected by the flow and keep their initial regularity
 \emph{even though the gradient of the velocity is not continuous}. The result  
 was strongly based on the observation that, for such solutions, the
 ``effective viscous  flux'' $F:=\nu\div u- P(\rho)$ is continuous, although singularities 
 persist in $\div u$ and $P(\rho)$ separately.  

 \medbreak
The present paper aims at completing the aforementioned works  in several directions. 

First, we want to supplement them with a uniqueness result. Indeed, in \cite{Hoff2}, D. Hoff 
constructed solutions (that are global in time  under some smallness assumption)
and pointed out very accurate qualitative properties for the geometric structure of singularities, 
but did not address  uniqueness.  That latter issue  has been considered afterward
in \cite{Hoff3}, but  only for \emph{linear} pressure laws. In fact, the main uniqueness theorem 
therein requires either the pressure law to be linear (as opposed to 
the standard isentropic assumption $P(\rho)=a\rho^\gamma$ with $\gamma>1$)
or some Lebesgue type information on $\nabla\rho$ (thus precluding our considering jumps across interfaces). 
To the best of our knowledge, exhibiting  an appropriate functional 
framework for uniqueness \emph{without imposing a special structure for the solutions} 
 has remained an open question for nonlinear pressure laws and discontinuous densities.

Our second goal is to extend Hoff's works concerning discontinuity across interfaces and uniqueness
 to any dimension and to more general pressure laws and density singularities.
\smallbreak
 Finding conditions on the initial data that ensure $\nabla u$ to belong to 
$L^1([0,T];L^\infty)$ for some $T>0$ is the key to our two goals.  That latter property 
will be  achieved by  combining   parabolic maximal regularity estimates with 
  \emph{tangential} (or striated) \emph{regularity} techniques  that are borrowed from the
work  by J.-Y. Chemin in \cite{Ch1991}.  
\medbreak
In order to introduce the reader with our use of maximal regularity,  let us consider 
the  slightly simpler situation of a fluid fulfilling 
the  inhomogeneous incompressible Navier-Stokes equations:
\beq\label{NSI}
\left\{
\begin{array}
{l} \displaystyle \pa_t\rho + \dive(\rho u) = 0\\[1ex]
\displaystyle \pa_t (\rho u) + \dive(\rho u \otimes u) - \mu \Delta u+ \nabla\Pi = 0\\[1ex]
\div u=0\,.
\end{array}
\right. \eeq
 We have in mind the case  where the initial density is given by  
\begin{equation}\label{eq:rho0}\rho_0=c_1\mathds{1}_{D}+c_2 \mathds{1}_{D^c}\end{equation}
for some positive constants $c_1,c_2$, with $D$ being a smooth bounded domain of $\R^d$ 
(above, $\mathds{1}_A$ designates the characteristic function of a set $A$).

Several recent works have been  devoted to proposing conditions on $u_0$ allowing 
for solving \eqref{NSI} either locally or globally, and uniquely.  The first result in that direction 
has been obtained  by the first author and P.B. Mucha in \cite{DM-cpam}. It 
 is based on endpoint maximal regularity and requires the jump  $c_1-c_2$ to be  small enough. 
However, that approach  requires the use of \emph{multiplier spaces} to handle the low regularity of the density,
and  is very unlikely to be extendable to the compressible setting,
owing to the pressure term that now depends on $\rho.$

Therefore,  we shall rather take advantage of the approach that has been proposed recently in \cite{HPZ} by  the third
author together with J. Huang and P. Zhang  to investigate 
 \eqref{NSI} with  only bounded density. 
 Indeed, it is based on the   standard parabolic maximal regularity and 
 requires only very elementary tools like H\"older inequality and Sobolev embedding. 
 In order to present  the main steps,  assume for simplicity that the  reference  density $\bar\rho$ is $1.$ 
  Then, setting $\vrho:=\rho-1,$   system \eqref{NSI} rewrites
$$\left\{\begin{array}{l} \displaystyle \pa_t\vrho + u\cdot\nabla\vrho  = 0\\[1ex]
\displaystyle \pa_t u - \mu \Delta u+ \nabla\Pi =  -\vrho\pa_t u -(1\!+\!\vrho) u\cdot\nabla u\\[1ex]
\div u=0\,.\end{array}\right.$$
{}From basic maximal regularity estimates (recalled in Subsection \ref{ss:max} below),  
we have\footnote{Throughout the paper
we agree that, whenever $E$ is a Banach space, $r\in[1,\infty]$  and $T>0$,
then $L_T^r(E)$ designates the space $L^r([0,T];E)$ and $\|\cdot\|_{L^r_T(E)}$ the
corresponding norm; when $T=\infty$, we use the notation $L^r(E)$.} for all $1<p,r<\infty,$ 
$$\|u\|_{L^\infty_T(\dot B^{2-2/r}_{p,r})}\,+ 
\,\left\|\left(\pa_tu,\mu\nabla^2u\right)\right\|_{L_T^r(L^p)}\,
\leq\, C\,\Bigl(\|u_0\|_{\dot B^{2-\frac2r}_{p,r}} \,
+\, \|\vrho\pa_t u\|_{L_T^r(L^p)} \,+\,\|(1\!+\!\vrho) u\cdot\nabla u\|_{L_T^r(L^p)}\Bigr)\cdotp$$
It is obvious that the second  term of the r.h.s. may be absorbed by the l.h.s. 
if the nonhomogeneity $\vrho$ is small enough for the $L^\infty$ norm. 
As for the last  term, it may be absorbed either for short time if the velocity is large, or for all times if the velocity is small, and
the norm $L^r(L^p)$ is \emph{scaling invariant} for the incompressible Navier-Stokes equations, 
that is to say  satisfies 
\begin{equation}\label{eq:critical}\frac2r+\frac dp=3.\end{equation}

Those simple observations are the keys to the proof of global existence 
for \eqref{NSI} in \cite{HPZ}. As regards uniqueness,  owing to the hyperbolic part of the system (viz. the first equation of \eqref{NSI}),  we need (at least) a $L^1_T({\rm Lip})$ control
on  the  velocity. Note that if one combines   the above control in $L_T^r(L^p)$ for $\nabla^2u$   with 
the corresponding critical Sobolev embedding, then  we miss  that information by a little
in the critical regularity setting, as having \eqref{eq:critical} and  $r>1$ implies that  $p<d$; nonetheless, it turns out that, if  working in a slightly subcritical framework
(that is $2/r+d/p<3$), one can get  existence and uniqueness altogether 
 for any initial density $\rho_0\in L^\infty$ that is close enough to some positive constant
 (see \cite{DPZ,HPZ} for more details).

In the particular case where $\rho_0$ is given by \eqref{eq:rho0} then, once the Lipschitz control of the transport field is available, it is possible to propagate
the Lipschitz regularity of  the domain $D$, as it  is just advected by the (Lipschitz continuous)
flow of the velocity field.  Based on that observation, 
further developments and more accurate informations on the evolution of
the boundary of $D$ have been obtained very recently by X. Liao and P. Zhang in \cite{LZ1,LZ2}, and by the first author with X. Zhang \cite{DXZ} and with P. B. Mucha \cite{DM}. 
In most of those works, a key ingredient is  the propagation of tangential  regularity, 
in the spirit of the seminal work by J.-Y Chemin  in \cite{Ch1991,Ch1993} dedicated to  the vortex patch 
problem for the incompressible Euler equations. 
We refer also to papers \cite{G-SR}, \cite{D1997}, \cite{D1999}, \cite{Hmidi} for extensions of the results of \cite{Ch1991,Ch1993} to higher dimensions and to
viscous homogeneous fluids; the case of non-homogeneous inviscid flows being treated  in \cite{F_2012}.
 
 \medbreak
 The rest of the paper is devoted to obtaining similar results
 for the compressible Navier-Stokes equations \eqref{NS}, and is structured as follows. 
  In the next section, we present our main results and  give some insight of the proofs. 
  Then, in Section \ref{s:tools}, we   recall the definition 
of Besov spaces and introduce  the  tools for achieving
our results: Littlewood-Paley decomposition, maximal regularity and estimates
involving striated regularity.
Section \ref{s:max-reg} is devoted to the proof of  our  main existence theorem for general discontinuous
densities while the next section concerns the propagation of striated regularity
and uniqueness. Some technical results that are based on harmonic analysis are postponed in the Appendix.


\section{Main results} \label{s:results}

In order to evaluate our chances of getting the same results
for \eqref{NS} as for \eqref{NSI} after suitable adaptation of
the method described above, let us rewrite \eqref{NS} in terms of $(\vrho,u).$ We get, just
denoting  by $P$ (instead of $P(1+\vrho)$) the pressure term,  the system
\begin{equation}\label{NSbis}\left\{\begin{array}
{l} \displaystyle \pa_t\vrho +  u\cdot\nabla\vrho+(\vrho+1)\dive u = 0\\[1ex]
\displaystyle  \pa_t u- \mu \Delta u-(\lambda+\mu)\nabla\dive u  = -\vrho\pa_tu  -(1\!+\!\vrho) u\cdot\nabla u - \nabla P\,.
\end{array}\right. \end{equation}
As the so-called Lam\'e system (that is, the l.h.s. of the second equation)  enjoys  the same 
 maximal regularity properties as the heat equation,  one can handle the 
 terms $\vrho\pa_tu$ and   $(1\!+\!\vrho) u\cdot\nabla u$ in a suitable $L^r(L^p)$ framework
 exactly as in the incompressible situation. However, by this method,  bounding 
 $\nabla P$ requires $\nabla\vrho$ to be in some Lebesgue space, a condition that we want to avoid. 
 In fact,   the coupling between the density and the velocity equations 
 is  stronger  than for \eqref{NSI} so that the  two equations of \eqref{NSbis} should not  be considered separately. 
For that reason, it is  much more difficult to prove a well-posedness result
 for rough densities here  than in the incompressible case, and the presence 
 of the ``out-of-scaling'' term $\nabla P$ precludes our achieving global existence  (even for small data)
 by simple  arguments. 
A standard way to weaken the coupling between the two equations of \eqref{NSbis}  (that has been used by D. Hoff in \cite{Hoff1,Hoff2} and, more recently, by B. Haspot 
 in \cite{Haspot} or by the first author with L. He in \cite{DH})
 is to reformulate the system in terms of a ``modified   velocity  field'',
 in the same spirit as the effective viscous flux mentioned in the introduction: we set
$$ w\,:=\,u\,+\,\nabla(-\nu\Delta)^{-1}P\,,\with\nu:=\lambda+\mu\,.$$
The modified velocity $w$ absorbs the ``dangerous part''
of $\nabla P$, as may be observed when writing out the equation for $w$: 
\begin{equation} \label{d:w0}(1+\vrho)\pa_tw-\mu\Delta w-\lambda\nabla\div w+ (1+\vrho)u\cdot\nabla u+(1+\vrho)(-\nu\Delta)^{-1}\nabla\pa_tP=0.\end{equation}
As, by virtue of the mass equation, we have  $$\pa_tP=(P-\rho P')\div u-\div(Pu);$$ 
 the last term of \eqref{d:w0}  is indeed of lower order and can be bounded in $L^r_T(L^p)$
whenever  $\vrho$ is bounded and belongs to  some suitable Lebesgue space. 

As we shall see in the present paper,  working with $w$ and $\vrho$ rather than with the original  unknowns $u$ and $\rho$ proves to be efficient if one is concerned with existence (and possibly uniqueness)
results for \eqref{NS} with only bounded density. 
In fact, we shall  implement the maximal regularity estimates 
on the equation fulfilled by $w,$ and   bound $\vrho$ by means of the standard a priori estimates
in Lebesgue spaces for the transport equation. 
For technical reasons however,
it will be wise  to replace $(-\nu\Delta)^{-1}$ by its non-homogeneous version $(\Id-\nu\Delta)^{-1}$ which is much less singular.

 That strategy will  enable us to prove the following local-in-time result of existence for \eqref{NS} 
 supplemented with  a rough initial density\footnote{The reader is referred to Section \ref{s:tools} below for the definition of the homogeneous Besov spaces $\dot B^s_{p,r}$.}.
\begin{thm}\label{thm:existence} 
Let $d\geq1$. Let the couple $(p,r)$ satisfy
\begin{equation}\label{eq:exposants}
d\,<\,p\,<\,\infty\qquad\hbox{ and }\qquad 1\,<\,r\,<\,\frac{2p}{2p-d}\,,
\end{equation}
and define the couple of indices $(r_0,r_1)$ by the relations
\begin{equation} \label{index:choice}
\frac1{r_0}\,=\,\frac1{r}\,-\,1\,+\,\frac d{2p}\qquad\mbox{ and }\qquad \frac1{r_1}\,=\,\frac1{r}\,-\,\frac12\cdotp
\end{equation}
Let the initial density $\rho_0$  and velocity $u_0$  satisfy:
\begin{itemize}
 \item $\vrho_0:=\rho_0-1$ in $\bigl(L^{p}\cap L^\infty\bigr)(\R^d)$;
\item $w_0:=u_0-v_0$  in  $\dot B^{2-2/r}_{p,r}$, with $v_0\,:=\,-\nabla(\Id-\nu\Delta)^{-1}(P(\rho_0))$. 
\end{itemize}

There exist $\veps>0$ and a time $T>0$ such that, if 
\begin{equation}\label{hyp:small-vrho}
\|\vrho_0\|_{L^\infty}\,\leq\,\veps\,,
\end{equation}
then there exists a solution $(\rho,u)$ to System \eqref{NS}-\eqref{eq:initial} on $[0,T]\times\R^d$ with $\vrho:=\rho-1$  satisfying
$$\|\vrho\|_{L^\infty([0,T]\times\R^d)}\leq 4\veps\andf \vrho \in \cC([0,T];L^q)\quad\hbox{for all }\ p\leq q<\infty, 
$$
 and $u=v+w$ with  $v\,:=\,-\nabla(\Id-\nu\Delta)^{-1}\bigl(P(\rho)\bigr)$ in $\cC\left([0,T];W^{1,q}(\R^d)\right)$ for all $p\leq q<\infty,$ and
$$
w\,\in\cC\bigl([0,T];\dot B^{2-2/r}_{p,r}\bigr)\cap \,L^{r_0}_T(L^{\infty}),\quad \nabla w\,\in\,L^{r_1}_T(L^{p})\andf \pa_tw\,,\;\nabla^2w\,\in\,L^{r}_T(L^{p})
$$
That solution is unique if $d=1.$
\medbreak

If,  in addition,   $u_0$ and $\vrho_0$ are in $L^2,$ and   $\inf P'>0$ on $[1-4\veps,1+4\veps],$  then the energy balance
\begin{multline}\label{est:en-est}
\frac12\left\|\sqrt{\rho(t)}\,u(t)\right\|^2_{L^2}\,+\,\left\|\Pi\bigl(\rho(t)\bigr)\right\|_{L^1}\,+  \\
+\,\mu\,\|\nabla u\|^2_{L^2_t(L^2)}\,+\,\lambda\,\|\div u\|^2_{L^2_t(L^2)}\,=\,
\frac12\left\|\sqrt{\rho_0}\,u_0\right\|^2_{L^2}\,+\,\left\|\Pi\bigl(\rho_0\bigr)\right\|_{L^1}
\end{multline}
holds true for all $t\in[0,T]$, where the  function $\Pi=\Pi(z)$ is defined by the conditions $\Pi(1)=\Pi'(1)=0$ and $\Pi''(z)\,=\,P'(z)/z$.
\end{thm} 

Combining the above statement with Sobolev embeddings  ensures 
that $\nabla w$ and $\div u$ are in $L^1_T(L^\infty).$
However, because operator $\nabla^2(\Id-\nu\Delta)^{-1}$ does not quite map $L^\infty$ into itself (except if $d=1$ of course), there is
no guarantee that the constructed velocity field $u$ has gradient in   $L^1_T(L^\infty)$. This seems to be the minimal requirement in order to get uniqueness of solutions
(see e.g. papers \cite{Hoff2} and \cite{D-Fourier}).
Keeping the model case \eqref{eq:rho0} in mind,  the question is whether 
adding up geometric hypotheses, like interfaces or tangential regularity, 
ensures that property and, hopefully, uniqueness. 

Motivated by the pioneering work by J.-Y. Chemin in \cite{Ch1991,Ch1993}, we shall assume that  the initial density has some ``striated regularity'' along a nondegenerate family
of vector fields. To be more specific,  we have to introduce more  notation and give some definitions. Before doing that, let us 
underline that   propagating tangential regularity for compressible
flows faces us to new difficulties  compared to the incompressible case, due to the fact that $\div u$ does not vanish anymore.

First of all, for any  $p$ in $]d,\infty]$, we denote by $\L^{\infty,p}$ the space of all
continuous and bounded  functions with gradient in $L^p(\R^d)$. 
Now, for a given vector-field $Y$ in  $\L^{\infty,p},$ 
we are interested in the regularity  of a function $f$ along $Y$, i.e. in the quantity
$$
\pa_Yf\,:=\,\,\sum_{j=1}^d\,Y^j\,\pa_jf\,.
$$
This expression is well-defined if $f$ is smooth enough, in which case we have the identity
\begin{equation} \label{eq:d_Yf}
 \pa_Yf\,=\,\div\bigl(f\,Y\bigr)\,-\,f\,\div Y\,.
\end{equation}
If  $f$ is only bounded (which, typically, will be the case if $f$ is the density given by \eqref{eq:rho0}), 
the above right-hand side  makes sense for any vector field $Y$  in $\L^{\infty,p},$
while $\pa_Yf$ has no meaning. Then 
 we take the right-hand side of  \eqref{eq:d_Yf} as a definition of $\pa_Yf$.


In order to define striated regularity, 
fix a family $\mc X=\left(X_\lambda\right)_{1\leq\lambda\leq m}$ of $m$ vector-fields with components in  $\L^{\infty,p}$ and suppose that it is  \emph{non-degenerate}, in the sense that
$$
I(\mc X)\,:=\,\inf_{x\in\mbb{R}^d}\,\sup_{\Lambda\in\Lambda^m_{d-1}}\,\left|\stackrel{d-1}{\wedge}X_\Lambda(x)\right|^{\frac{1}{d-1}}
\,\,\,>\,0\,.
$$
Here $\Lambda\in\Lambda^m_{d-1}$ means that $\Lambda=\left(\lambda_1,\ldots,\lambda_{d-1}\right)$, with each
$\lambda_i\in\left\{1,\ldots,m\right\}$ and $\lambda_i<\lambda_j$ for $i<j$, while the symbol
$\stackrel{d-1}{\wedge}X_\Lambda$ stands for the unique element of $\R^d$ such that
$$
\forall\,\,Y\in\R^d\,,\qquad\Bigl(\stackrel{d-1}{\wedge}X_\Lambda\Bigr)\cdot Y\,=\,
\det\left(X_{\lambda_1}\ldots X_{\lambda_{d-1}},Y\right)\,.
$$
Then we set
$$
\|X_\lambda\|_{\L^{\infty,p}}\,:=\,\|X_\lambda\|_{L^\infty}\,+\,\|\nabla X_\lambda\|_{L^p}
\andf \tn{\mc X}_{\L^{\infty,p}}\,:=\,\sup_{\lambda\in\Lambda} \|X_\lambda\|_{\L^{\infty,p}}\,.
$$
More generally, whenever $E$  is a normed  space, we  use the notation 
$\tn{\mc X}_E\,:=\, \sup_{\lambda\in\Lambda} \|X_\lambda\|_E$.

\begin{defi} \label{d:stri}
Take a vector-field $Y\,\in\,\L^{\infty,p}$, for some $p\in]d,\infty]$.
A function $f\in L^\infty$ is said to be of class $L^p$ along $Y$, and we write $\,f\in\L^p_Y$, if
$\;\div\left(f\,Y\right)\,\in\,L^p\left(\R^d\right)$.

If  $\mc X=\left(X_\lambda\right)_{1\leq\lambda\leq m}$ is a non-degenerate family of vector-fields in $\L^{\infty,p}$ then we
set
$$
\L^p_{\mc X}\,:=\,\bigcap_{1\leq\lambda\leq m}\,\L^p_{X_\lambda}\qquad\mbox{ and }
\qquad
\left\|f\right\|_{\L^{p}_{\mc X}}\,:=\,\frac{1}{I(\mc X)}\,\Bigl(\|f\|_{L^\infty}\,\tn{\mc X}_{\L^{\infty,p}}\,+\,\tn{\div\left(f\,\mc X\right)}_{L^p}\Bigr)\cdotp
$$
\end{defi}

The main motivation for Definition \ref{d:stri} is Proposition \ref{p:tang} below that 
states that, if $\vrho$ is bounded  and if, 
in addition,  $\vrho\in \L^p_{\mc X}$ for some non-degenerate family $\cX$ of vector-fields
in $\L^{\infty,p}$ with $d<p<\infty,$ then $(\eta\Id-\Delta)^{-1}\nabla^2f(\vrho)$ is in $L^\infty$  for all $\eta>0$ and smooth enough function $f.$  
That fundamental property  will enable us to consider  data like \eqref{eq:rho0} in a functional framework
that ensures persistence of interface regularity and uniqueness altogether, as stated just below. 
\begin{thm}\label{thm:reg}
Let $d\geq1$ and  the couple $(p,r)$ fulfill  conditions \eqref{eq:exposants}. Consider initial data $(\rho_0,u_0)$ satisfying the same assumptions as in Theorem \ref{thm:existence}. 
Assume in addition that  there exists  a non-degenerate family 
  $\mc X_0\,=\,\bigl(X_{0,\la}\bigr)_{1\leq\la\leq m}$ of vector-fields in $\L^{\infty,p}$ such that
$\rho_0$ belongs to  $\mbb L^{p}_{\mc X_0}$.

Then, there exists a time $T>0$ and a unique solution  $(\rho,u)$ to System \eqref{NS}-\eqref{eq:initial} on $[0,T]\times \R^d$, such that
$\vrho:=\rho-1$, $v:=-\nabla(\Id-\nu\Delta)^{-1}\bigl(P(\rho)\bigr)$ and $w:=u-v$ satisfy the same properties as in Theorem \ref{thm:existence}.
Furthermore,  $\nabla u$ belongs to $L^1\bigl([0,T];L^\infty(\R^d)\bigr)$ and 
 $u$ has a   flow $\psi_u$ with bounded gradient, that is the unique solution of \begin{equation}
\psi_u(t,x)\,=\,x\,+\,\int_0^t u\bigl(\tau,\psi_u(\tau,x)\bigr)\,d\tau\qquad\hbox{for all }\ (t,x)\in[0,T]\times\R^d\,.
\end{equation}
Finally, if we define $X_{t,\lambda}$ by the formula 
$X_{t,\lambda}(x):=\pa_{X_{0,\lambda}}\psi_u\bigl(t,\psi_u^{-1}(t,x)\bigr)$
then, for all time $t\in[0,T]$, the family  $\mc X_t:=\bigl(X_{t,\lambda}\bigr)_{1\leq\lambda\leq m}$ remains  non-degenerate  and   in the space $\L^{\infty,p}$,
and the density $\rho(t)$ belongs to  $\mbb L^{p}_{\mc X_t}$.
\end{thm}


Let us make some comments on the proof of that second main result.
As pointed out above, the striated regularity hypothesis ensures
that $(\Id-\nu\Delta)^{-1}\nabla^2P$ is bounded. 
Since we know from Theorem \ref{thm:existence}  that $\nabla w$ is in 
$L^1_T(L^\infty),$ one can conclude that also $\nabla u$ belongs to $L^1_T(L^\infty)$.
From this property and the remark that, for all $\lambda\in\Lambda$, one has 
 $$ \pa_tX_\lambda+u\cdot\nabla X_\lambda =\pa_{X_\lambda}u\quad\hbox{and}\quad
 \pa_t\div(\rho\,X_\lambda)\,+\,\div\bigl(\div(\rho\,X_\lambda)\,u\bigr)\,=\,0\,,$$
 standard estimates for the transport equation  will enable us  to  propagate  the tangential regularity.

As regards the proof of uniqueness, we adopt 
D. Hoff viewpoint in \cite{Hoff3}: \textit{solutions with minimal regularity are best if compared in 
a Lagrangian framework; that is, we compare the instantaneous states of corresponding fluid particles in 
two different solutions 
rather than the states of different fluid particles instantaneously occupying the same point
of space-time.}
However, the proof that is proposed therein relies on stability estimates in the negative 
Sobolev space $\dot H^{-1}$ for the density; at some point,
 it is crucial that $\rho\in\dot H^{-1}$ implies that $P(\rho)$ is in $\dot H^{-1},$ too, a property that
 obviously fails when the pressure law is nonlinear. 

 In the present paper, adopting  the Lagrangian viewpoint  will enable us  to avoid
 (for   general pressure laws)  the  loss of one derivative due to the hyperbolic part of System \eqref{NS}.
  As a matter of fact, we shall establish  stability estimates for the (Lagrangian) velocity field 
directly  in the energy space, and  the presence of variable coefficients owing  to the initial density variations,
either in front of the time derivative or in the elliptic part of the evolution operator, will be  harmless.
 \medbreak
 Let us finally state an important  application of Theorem \ref{thm:reg}.
 \begin{col} Assume that $\rho_0$ is given by \eqref{eq:rho0} for some bounded domain $D$
of class $W^{1,p}$ with $d<p<\infty$.

Then, there exists $\veps>0$ such that, if  $|c_1-c_2|\leq\veps$, then for any initial velocity
 field $u_0$ satisfying the conditions of Theorem \ref{thm:existence}, there exists $T>0$
 such that System \eqref{NS}-\eqref{eq:initial} admits a unique solution $(\rho,u)$. Furthermore, the density at time
 $t$ has a jump discontinuity along the interface of the domain $D_t$ transported by the flow 
 of $u,$ and $\pa D_t$ keeps its $W^{1,p}$ regularity.
 \end{col}
We end this section with a list of possible extensions/improvements of our paper.
\begin{enumerate}
\item All our results may be readily adapted to the case of periodic boundary conditions; indeed, 
our techniques rely on Fourier analysis and thus hold true for functions defined on the torus.
\item We expect  a similar  existence statement if the fluid domain is a bounded open set $\Omega$ with (say) $\cC^2$ 
boundary, and the system is supplemented with homogeneous Dirichlet boundary 
conditions for the velocity. 
Indeed, in that setting, the Besov spaces may be defined by real interpolation 
from the domain of the Lam\'e (or, equivalently the heat) operator
and the maximal regularity  estimates remain the same. The reader
may refer to  \cite{D-bounded} for an example of solving \eqref{NS} 
in that setting, in the case of density in $W^{1,p}$ for some $p>d.$ 

Concerning  the propagation of tangential regularity, the situation where
the reference family of vector-fields does not degenerate at the boundary should be 
tractable with few  changes (this is a matter of adapting the work by N. Depauw in \cite{Depauw}
to our system). This means that one can  consider initial densities like \eqref{eq:rho0}
provided the boundary of $D$ does not meet that of $\Omega.$  
\item To keep the paper a reasonable size, we refrained from considering the global existence
issue for rough densities. We plan to address that interesting question in the near future.
\end{enumerate}


\section{Tools} \label{s:tools}
Here we introduce the main tools for our analysis. First of all, we recall 
 basic facts about Littlewood-Paley theory and Besov spaces. The next subsection is devoted to  maximal regularity results. Finally, in Subsection \ref{ss:tang} we present 
key inequalities involving  striated regularity.

\subsection{Littlewood-Paley theory and Besov spaces} \label{ss:LP}

We  here briefly present the  Littlewood-Paley theory, 
as it will come into play  for proving our main result. We refer e.g. to Chapter 2 of \cite{B-C-D} for more details.
For simplicity of exposition, we focus on  the $\R^d$ case; however, the whole construction can be adapted  to the $d$-dimensional torus $\T^d$.

\medbreak
First of all, let us introduce the so-called ``Littlewood-Paley decomposition''. It is based on a non-homogeneous dyadic partition of unity with respect to the Fourier variable: 
 fix a smooth radial function $\chi$ supported in the ball $B(0,2)$, equal to $1$ in a neighborhood of $B(0,1)$
and such that $r\mapsto\chi(r\,e)$ is nonincreasing over $\R_+$ for all unitary vectors $e\in\R^d$. Set
$\varphi\left(\xi\right)=\chi\left(\xi\right)-\chi\left(2\xi\right)$ and
$\vphi_j(\xi):=\vphi(2^{-j}\xi)$ for all $j\geq0$.
The non-homogeneous dyadic blocks $(\Delta_j)_{j\in\Z}$ are defined by\footnote{Throughout we agree  that  $f(D)$ stands for 
the pseudo-differential operator $u\mapsto\mc{F}^{-1}(f\,\mc{F}u)$.} 
$$
\Delta_j\,:=\,0\quad\mbox{ if }\; j\leq-2,\qquad\Delta_{-1}\,:=\,\chi(D)\qquad\mbox{ and }\qquad
\Delta_j\,:=\,\varphi(2^{-j}D)\quad \mbox{ if }\;  j\geq0\,.
$$
We  also introduce the following low frequency cut-off operator:
\begin{equation} \label{eq:S_j}
S_j\,:=\,\chi(2^{-j}D)\,=\,\sum_{k\leq j-1}\Delta_{k}\quad\mbox{ for }\quad j\geq0,\quad
\hbox{and }\ S_j=0\ \hbox{ for }\ j<0.
\end{equation}
It is well known that for any $u\in\mc{S}'$,  one has the equality 
$$u\,=\,\sum_{j\geq-1}\Delta_ju\qquad\mbox{ in }\quad \mc{S}'\,.$$

Sometimes, we shall alternately use the following 
spectral cut-offs $\dot\D_j$ and $\dot S_j$ that are  defined by
$$
\dot\Delta_j\,:=\,\varphi(2^{-j}D)\andf \dot S_j\,=\,\chi(2^{-j}D)\  \mbox{ for all }\  j\in\Z\,.
$$
Note  that  we have 
\begin{equation}\label{eq:LPh} u\,=\,\sum_{j\in\Z}\dot\Delta_ju\end{equation}
up to  polynomials only, which makes decomposition  \eqref{eq:LPh} unwieldy. A way to have equality in \eqref{eq:LPh} in the sense of tempered distributions 
is to restrict oneself to elements $u$  of  the set $\mc S_h'$ of tempered distributions such that
$$\lim_{j\ra-\infty}\bigl\|\dot S_ju\bigr\|_{L^\infty}\,=\,0\,.$$

It is now time to introduce  Besov spaces.
\begin{defi} \label{d:B}
Let $s\in\R$ and $1\leq p,r\leq\infty$.
\begin{enumerate}[(i)]
\item The \emph{non-homogeneous Besov space}
$B^{s}_{p,r}$ is the  set of tempered distributions $u$ for which
$$
\|u\|_{B^{s}_{p,r}}\,:=\,
\left\|\left(2^{js}\,\|\Delta_ju\|_{L^p}\right)_{j\geq-1}\right\|_{\ell^r}\,<\,\infty\,.
$$
\item The \emph{homogeneous Besov space} $\dot B^s_{p,r}$ is the subset of distributions $u$ in $\mc S'_h$ such that
$$\|u\|_{\dot B^{s}_{p,r}}\,:=\,\left\|\left(2^{js}\,\|\dot\Delta_ju\|_{L^p}\right)_{j\in\Z}\right\|_{\ell^r}\,<\,\infty\,.$$
\end{enumerate}
\end{defi}

It is well known that  $B^s_{2,2}$ coincides with $H^s$ (with equivalent norms)
and that nonhomogeneous (resp. homogeneous) 
Besov spaces are interpolation spaces between Sobolev spaces $W^{k,p}$ (resp. $\dot W^{k,p}$). 
Furthermore, for all $p\in]1,\infty[$, one has the following continuous embeddings (see the proof in \cite[Chap. 2]{B-C-D}): 
$$\dot B^0_{p,\min(p,2)}\hookrightarrow L^p\hookrightarrow \dot B^0_{p,\max(p,2)}
\andf B^0_{p,\min(p,2)}\hookrightarrow L^p\hookrightarrow B^0_{p,\max(p,2)}.$$
We shall also often use the  embeddings that are stated in the following proposition. 
\begin{prop}\label{p:embed}
Let $1\leq p_1\leq p_2\leq\infty.$ 
The space $B^{s_1}_{p_1,r_1}(\R^d)$ is continuously embedded in the space $B^{s_2}_{p_2,r_2}(\R^d),$ if
$$
s_2\,<\,s_1-d\left(\frac{1}{p_1}-\frac{1}{p_2}\right)\qquad\mbox{ or }\qquad
s_2\,=\,s_1-d\left(\frac{1}{p_1}-\frac{1}{p_2}\right)\;\;\mbox{ and }\;\;r_1\,\leq\,r_2\,. 
$$
The space $\dot B^{s_1}_{p_1,r_1}(\R^d)$ is continuously embedded in the space $\dot B^{s_2}_{p_2,r_2}(\R^d)$ if
$$
s_2\,=\,s_1-d\left(\frac{1}{p_1}-\frac{1}{p_2}\right)\;\;\mbox{ and }\;\;r_1\,\leq\,r_2\,. 
$$
\end{prop}
Finally, we shall need the following continuity result.
\begin{lem}\label{l:D-1} There exists a constant $C$, depending only on $d$, such that for all $p\in[1,\infty],$ we have
$$\|\Delta(\Id-\Delta)^{-1}f\|_{L^p}\leq C\|f\|_{L^p}.
$$
\end{lem}
\begin{proof}
It suffices to notice that 
$$\Delta(\Id-\Delta)^{-1}f=(\Id-\Delta)^{-1}f-f$$
and that $(\Id-\Delta)^{-1}$ maps $L^p$ to $B^{2}_{p,\infty}$ (see Proposition 2.78 of  \cite{B-C-D}) 
hence to $L^p,$ with a constant independent of $p$.
\end{proof}

\begin{col} \label{c:Delta-Id}
Let $u$ solve the elliptic equation $(\Id-\D)u\,=\,f$ in $\R^d$, with $f\in L^p$ for some $p\in[1,\infty]$.
Then $u\in W^{1,p}(\R^d)$, with $\Delta u\in L^p$, and one has the estimate
$$
\|u\|_{W^{1,p}}\,+\,\|\Delta u\|_{L^p}\,\leq\,C\,\|f\|_{L^p}\,,
$$
for some positive constant $C$ depending just on $d$. 
\medbreak
If  $1<p<\infty$, then  $u\in W^{2,p}$ and  we have
$$\|u\|_{W^{2,p}}\,\leq\,C\,\|f\|_{L^p}\,.$$
\end{col}

\begin{proof}
From Lemma \ref{l:D-1}, we gather that $\Delta u\in L^p$, hence $u\,=\,\Delta u\,+\,f$ belongs to $L^p$, too. This relation in particular implies the control
$\|u\|_{L^p}\leq C\,\|f\|_{L^p}$.
At this point, the control of the gradient of $u$ in $L^p$ follows e.g. from Gagliardo-Nirenberg inequalities (or a decomposition into low and high frequencies).
Finally, in the case $1<p<\infty$, the fact that $\Delta u\in L^p$ implies that $\nabla^2u\in L^{p}$ by Calder\'on-Zygmund theory.
\end{proof}


\subsection{Maximal regularity and propagation of $L^p$ norms} \label{ss:max}

In this subsection we recall some results about maximal regularity for the heat equation, then extend them to the elliptic operator $\mc L$ defined in \eqref{def:L}.
Those results will be essentially the key to  Theorem \ref{thm:existence}, namely existence of solutions in an $L^p$ setting.

\subsubsection{The case of the heat kernel} \label{sss:max-D}
Here we  focus on maximal regularity results for the  heat semi-group, 
as they  will entail similar ones for the Lam\'e semi-group generated by  $-\mc L$ (see Paragraph \ref{sss:max-L} below).
We first look at the propagation of regularity for the initial datum.
Our starting point is the proposition below, that  corresponds to Theorem 2.34 of \cite{B-C-D}.
\begin{prop} \label{prop1.1}
Let $s>0$ and $(p,r)\in [1,\infty]^2$.
A constant $C$ exists such that
\beno
C^{-1}\,\|z\|_{\dot{B}^{-s}_{p,r}}\,\leq\,
\bigl\|\|t^{s/2}\,e^{t\D}z\|_{L^p}\bigr\|_{L^r(\R_+;\f{dt}{t})}\,\leq\,
C\,\|z\|_{\dot{B}^{-s}_{p,r}}\,.
\eeno
\end{prop}

Thus we deduce   that,  for all $r\in[1,\infty[\,$, having $z$ in $\dot{B}^{-2/r}_{p, r}(\R^d)$ is equivalent to the condition
$e^{t\D}z\in L^r\bigl(\R_+;L^p(\R^d)\bigr)$.
In particular, taking $z=\Delta u_0$ and assuming that $u_0$ has  
 \emph{critical regularity}  $\dot{B}^{-1+d/p}_{p,r}(\R^d)$ for some $p\in(1,\infty)$ and  $r$ fulfilling
 \begin{equation}\label{eq:prcrit}
2-\frac2r=\frac dp-1\quad\hbox{with }\  1 < r<\infty,
 \end{equation}
  Proposition \ref{prop1.1} combined with the classical $L^p$ theory for the Laplace operator imply  that  $\nabla^2  e^{t\D}u_0$ is in 
$L^{r}\bigl(\R_+; L^{p}(\R^d)\bigr).$ 
\medbreak
Note however that \eqref{eq:prcrit}  
   gives the constraint $d/3<p\leq d,$
which   is too restrictive for our scopes:  we will need $p>d$ in order to guarantee that 
$\nabla u$ is in $L^1\bigl([0,T];L^\infty(\R^d)\bigr)$ for some $T>0$ (see Section \ref{s:max-reg}
for more details). This fact will preclude us from working in the critical regularity setting. 

Before going on, let us introduce more notations:  throughout this section,  we will use 
an index $j$ to designate the regularity of Lebesgue exponents   $(p_j,r_j)\in\,[1,\infty]^2$  pertaining to the term $\nabla^jh$.

According to Proposition \ref{prop1.1},  if  $u_0$ is in $\dot B^{s_2}_{p_2,r_2}$ with  $s_2\,=\,2-2/r_2$ and $1<p_2,r_2<\infty,$  then
\begin{equation} \label{est:nabla^2-h}
\nabla^2 e^{t\Delta}u_0\,\in\,L^{r_2}\bigl(\R_+;L^{p_2}(\R^d)\bigr)\cdotp
\end{equation}
Furthermore, we have 
\begin{equation}\label{est:ch}
e^{t\Delta}u_0\in\cC_b(\R_+;\dot B^{s_2}_{p_2,r_2})\end{equation} since 
$$\|e^{t\Delta}u_0\|_{\dot B^{2-\frac2{r_2}}_{p_2,r_2}}\sim \|\nabla^2e^{t\Delta}u_0\|_{\dot B^{-\frac2{r_2}}_{p_2,r_2}}
\sim \biggl(\int_{\R_+} \big\|e^{\tau\Delta}(\Delta e^{t\Delta} u_0)\big\|_{L^{p_2}}^{r_2}\,d\tau\biggr)^{\frac1{r_2}}$$ 
and, using the fact that the heat semi-group is contractive on $L^{p_2},$ 
$$
\int_{\R_+} \big\|e^{\tau\Delta}(\Delta e^{t\Delta} u_0)\big\|_{L^{p_2}}^{r_2}\,d\tau
\leq \int_{\R_+} \big\|e^{\tau\Delta}\Delta u_0\big\|_{L^{p_2}}^{r_2}\,d\tau \leq C\| \Delta u_0\|_{\dot B^{2-\frac2{r_2}}_{p_2,r_2}}^{r_2}.
$$
Time continuity in \eqref{est:ch}  just follows from  the fact that  $\cS$ is densely embedded in $L^{p_2}.$
\medbreak
Next, by the embedding properties of Proposition \ref{p:embed}, we  have, if  $p_1\geq p_2$ and $r_1\geq r_2$,
$$
\nabla u_0\,\in\,\dot B^{s_1}_{p_1,r_1}\qquad\mbox{ with }\qquad s_1\,=\,1-\frac{2}{r_2}-d\left(\frac{1}{p_2}-\frac{1}{p_1}\right)\cdotp
$$
In order to be in position of applying Proposition \ref{prop1.1} so as to get that 
$\nabla e^{t \Delta}u_0\,\in\,L^{r_1}\bigl(\R_+;L^{p_1}(\R^d)\bigr),$
 we  need to have  in addition  $s_1<0,$  that is to say
\begin{equation} \label{cond:r_2-p_0}
\frac{2}{r_2}\,+\,\frac{d}{p_2}\,-\,\frac{d}{p_1}\,>\,1.
\end{equation}
Then, defining $r_1$ by 
\begin{equation} \label{cond:HLS_1}
\frac{2}{r_2}\,+\,\frac{d}{p_2}\;=\;1\,+\,\frac{2}{r_1}\,+\,\frac{d}{p_1}, 
\end{equation}
  we get 
\begin{equation} \label{est:nabla-h}
\nabla e^{t \Delta}u_0\,\in\,L^{r_1}\bigl(\R_+;L^{p_1}(\R^d)\bigr)\cdotp
\end{equation}
Finally, let us consider  $e^{t\Delta}u_0.$
By critical embedding, we have, if $p_0\geq p_2$ and $r_0\geq r_2,$ 
$$
u_0\,\in\,\dot B^{s_0}_{p_0,r_0}\qquad\mbox{ with }\qquad s_0\,=\,2-\frac{2}{r_2}-d\left(\frac{1}{p_2}-\frac{1}{p_0}\right)\cdotp
$$
If we want to resort again to Proposition \ref{prop1.1},  we need to have in addition that 
\begin{equation} \label{cond:r_2-p_1}
s_0\,=\,2\,-\,\frac{2}{r_2}\,-\,\frac{d}{p_1}\,<\,0\,.
\end{equation}
Under that condition, choosing $r_0$ so that  $s_0\,=\,-2/r_0,$ that is to say such that
 \begin{equation} \label{cond:HLS_0}
\frac{2}{r_2}\,+\,\frac{d}{p_2}\;=\;2\,+\,\frac{2}{r_0}\,+\,\frac{d}{p_0},
\end{equation}
 we end up with 
\begin{equation} \label{est:h}
e^{t \Delta}u_0\,\in\,L^{r_0}\bigl(\R_+;L^{p_0}(\R^d)\bigr)\,.
\end{equation}
Let us next consider the propagation of regularity for the forcing term in the heat equation. We start by presenting the standard  maximal $L^r(L^p)$ regularity for the heat semi-group
(see the proof in  \cite[Lem. 7.3]{Lem} for instance).
\begin{lem}\label{l:max-reg_2}
Let us define the operator $\cA_2$ by the formula
$$
\mc A_2\,:\quad f\,\longmapsto\,\int_0^t \nabla^2 e^{(t-s)\Delta}f(s,\cdotp)\,ds\,.
$$

Then $\mc A_2$ is bounded from  $L^{r_2}\bigl(\,]0,T[\,; L^{p_2}(\R^d)\bigr)$ to $L^{r_2}\bigl(\,]0,T[\,; L^{p_2}(\R^d)\bigr)$ for every $T\in\,]0,\infty]$ and $1\,<\,p_2,r_2\,<\,\infty$.
Moreover, there holds:
\beno \|\cA_2 f\|_{L^{r_2}_T(L^{p_2})}\,\leq\,C\|f\|_{L^{r_2}_T(L^{p_2})}\,.\eeno
\end{lem}

As for the propagation of regularity for the first derivatives, we have the following statement. 
\begin{lem}\label{l:max-reg_1} 
Assume that the Lebesgue exponents $p_1$ and $p_2$ fulfill $0\leq 1/p_2-1/p_1<1/d,$ and that $1<r_2 < r_1< \infty$ are  interrelated by the relation  \eqref{cond:HLS_1}. Let us define the operator $\cA_1$ by
$$
\mc A_1\,:\quad f\,\mapsto\,\int_0^t \nabla e^{(t-s)\Delta}f(s,\cdotp)\,ds\,.
$$

Then $\mc A_1$ is bounded from  $L^{r_2}\bigl(\,]0,T[\,; L^{p_2}(\R^d)\bigr)$ to $L^{r_1}\bigl(\,]0,T[\,; L^{p_1}(\R^d)\bigr)$ for every $T\in\,]0,\infty]$, and there holds
\beno \|\cA_1 f\|_{L^{r_1}_T(L^{p_1})}\,\leq\,C\,\|f\|_{L^{r_2}_T(L^{p_2})}\,,
\eeno
for a suitable constant $C>0$ depending only on the space dimension $d\geq1$ and on $p_1,r_1,p_2,r_2$.
\end{lem}

\begin{proof}
We use the fact that for all $0\leq s\leq t\leq T,$
$$\begin{aligned}
 \na e^{(t-s)\Delta}f(s,x)\,&=\,\f{\sqrt{\pi}}{(4\,\pi\,(t-s))^{\f{d+1}2}}\int_{\R^d}\f{(x-y)}{2\,\sqrt{(t-s)}}\,
\exp\Bigl(-\f{|x-y|^2}{4\,(t-s)}\Bigr)\,f(s,y)\,dy \\
&\eqdefa\, \f{\sqrt{\pi}}{(4\,\pi\,(t-s))^{\f{d+1}2}}\,
K_1\biggl(\f{\cdot}{\sqrt{4\,(t-s)}}\biggr)\ast_x f(s,\cdot)\,. \end{aligned}$$

Applying Young's inequality in the space variables yields, 
\begin{align*}
\left\|\na e^{(t-s)\Delta}\,\left(\mathds{1}_{[0,T]}\,f\right)(s,\cdot)\right\|_{L^{p_1}}\,&\leq\,
C((t-s))^{-(d+1)/2}\biggr\|K_1\biggl(\f{\cdot}{\sqrt{4\pi(t-s)}}\biggr)\biggr\|_{L^{m_1}}\,\|\mathds1_{[0,T]}\,f\|_{L^{p_2}} \\
&\leq\,C((t-s))^{-\beta}\,\|\mathds1_{[0,T]}\,f\|_{L^{p_2}}\,,
\end{align*}
where we have defined $m_1$ and $\beta$ by  
$$\frac1{p_1}\,+\,1\,=\,\frac1{m_1}\,+\,\frac1{p_2}\quad\hbox{and}\quad\beta\,=\,\frac{d+1}2\,-\,\frac d{2m_1}
\,=\,\frac12\,+\,\frac d{2p_2}\,-\,\frac d{2p_1}\cdotp$$
Note that the conditions in the lemma ensure that $\beta\in[1/2,1[.$ 
At this point, we apply Hardy-Littlewood-Sobolev inequality (see e.g. Theorem 1.7 of \cite{B-C-D}) with respect to time: since $r_1$ and $r_2$ verify
$1/r_1\,+\,1\,=\,1/r_2\,+\,\beta$ by hypothesis \eqref{cond:HLS_1}, we immediately get the claimed inequality.
The lemma is thus proved.
\end{proof}

We now state integrability properties concerning $f$ itself, without taking any derivative.
\begin{lem}\label{l:max-reg_0} Assume that the Lebesgue exponents $p_0$ and $p_2$ fulfill $0\leq 1/p_2-1/p_0<2/d,$ and that $1<r_2 < r_0< \infty$ 
 are  interrelated by the relation \eqref{cond:HLS_0}.
Define the operator $\cA_0$ by the formula
$$
\mc A_0\,:\quad f\,\longmapsto\,\int_0^t e^{(t-s)\Delta}f(s,\cdotp)\,ds\,.
$$

Let $s_2:=2-2/r_2.$ Then $\mc A_0$ is bounded from  $L^{r_2}\bigl(\,]0,T[\,; L^{p_2}\bigr)$ to $L^{r_0}\bigl(\,]0,T[\,; L^{p_0}\bigr)\times \cC([0,T];\dot B^{s_2}_{p_2,r_2})$ for every $T\in\,]0,\infty]$, and there holds
\beno   \|\cA_0f\|_{L^\infty_T(\dot B^{s_2}_{p_2,r_2})}+ \|\cA_0 f\|_{L^{r_0}_T(L^{p_0})}\,\leq\,C\,\|f\|_{L^{r_2}_T(L^{p_2})}\,,
\eeno
for a suitable constant $C>0$ depending  on the space dimension $d\geq1$ and on $p_0, r_0,p_2,r_2$.
\end{lem}

\begin{proof}
 The proof of the continuity in $L^{r_0}(]0,T[;L^{p_0})$ goes as in Lemma \ref{l:max-reg_1}:   we start by writing
$$\begin{aligned}
e^{(t-s)\Delta}f(s,x)\,&=\,\f{1}{(4\,\pi\,(t-s))^{\f{d}2}}\int_{\R^d}\exp\Bigl(-\f{|x-y|^2}{4\,(t-s)}\Bigr)\,f(s,y)\,dy \\
&\eqdefa\, \f{1}{(4\,\pi\,(t-s))^{\f{d}2}}\;K_0\biggl(\f{\cdot}{\sqrt{4\,(t-s)}}\biggr)\ast_x f(s,\cdot)\,. \end{aligned}
$$
Then, we apply Young's inequality in the space variables and get, for every $s>0$ fixed,
\begin{align*}
\left\|e^{(t-s)\Delta}\,\left(\mathds{1}_{[0,T]}\,f\right)(s,\cdot)\right\|_{L^{p_0}}\,&\leq\,
C((t-s))^{-d/2}\,\biggl\|K_0\biggl(\f{\cdot}{\sqrt{4\pi(t-s)}}\biggr)\biggr\|_{L^{m_0}}\,\|\mathds1_{[0,T]}\,f\|_{L^{p_2}} \\
&\leq\,C((t-s))^{-\g}\,\|\mathds1_{[0,T]}\,f\|_{L^{p_2}}\,,
\end{align*}
where, exactly as before, we have defined $m_0$ and $\g$ by the relations 
$$1+\frac1{p_0}\,=\,\frac1{m_0}\,+\,\frac1{p_2}\quad\hbox{and}\quad
\g=\frac d{2p_2}-\frac d{2p_0}\cdotp$$
Our assumptions ensure that  $\g\in\,]0,1[\,$ and one may  apply  Hardy-Littlewood-Sobolev inequality with respect to time. Since $r_0$ and $r_2$ verify
$1/r_0\,+\,1\,=\,1/r_2\,+\,\g$ by hypothesis \eqref{cond:HLS_0}, we immediately get  that  $\cA_0$ is bounded 
from $L^{r_2}_T(L^{p_2})$ to $L^{r_0}_T(L^{p_0}).$ 
\medbreak 
The second part of the statement is classical. Arguing by density, it suffices to establish that 
$$   \|\cA_0f(T)\|_{\dot B^{s_2}_{p_2,r_2}}\,\leq\,C\,\|f\|_{L^{r_2}_T(L^{p_2})}.$$
To this end, we write, using Proposition \ref{prop1.1} and  an obvious change of variable,  that 
$$\begin{aligned}
  \|\cA_0f(T)\|_{\dot B^{s_2}_{p_2,r_2}} &\sim \biggl(\int_{0}^\infty \big\|\Delta e^{ t\Delta} \cA_0f(T)\bigr\|_{L^{p_2}}^{r_2}\,dt\biggr)^{\frac1{r_2}}\\
 &\sim \biggl(\int_{0}^\infty \bigg\|\int_0^T\Delta e^{ (t+T-\tau)\Delta} f(\tau)\,d\tau\biggr\|_{L^{p_2}}^{r_2}\,dt\biggr)^{\frac1{r_2}}\\
 &\leq C \biggl(\int_{T}^\infty \bigg\|\int_0^{t'}\Delta e^{(t'-\tau)\Delta} (f\mathds{1}_{[0,T]})(\tau)\,d\tau\biggr\|_{L^{p_2}}^{r_2}\,dt\biggr)^{\frac1{r_2}}\cdotp
\end{aligned}$$
Then using Lemma \ref{l:max-reg_2} to bound the right-hand side yields the claimed inequality.\end{proof}


\subsubsection{Maximal regularity results for operator $\mc L$} \label{sss:max-L}
Here we want to extend the results of the previous  paragraph to the elliptic operator  $\mc L\,=\,-\mu\Delta-\lambda\nabla\div,$ under Condition \eqref{lame}:
for suitable initial datum $h_0$ and external force $f$,  let us consider the equation
\begin{equation}\label{eq:lame}
\left\{\begin{array}{l}
\pa_t h\,+\,\cL h\,=\,f \\[1ex]
h_{|t=0}\,=\,h_0\,.
\end{array}
\right.
\end{equation}
The following statement will be a key ingredient in the proof of our existence result.
\begin{prop} \label{p:max-reg}
Let $\bigl((p_j,r_j)\bigr)_{j=0,1,2}$ 
satisfy  $1<p_2,r_2<\infty,$ $r_2<r_0,$ $r_2<r_1,$ $p_0\geq p_2,$ $p_1\geq p_2,$ and
the relations \eqref{cond:HLS_1} and \eqref{cond:HLS_0}.  Let  $h_0$ be in $\dot B^{s_2}_{p_2,r_2}$ with $s_2:=2-2/r_2,$ and
let $f$ be in $L^{r_2}_{\rm loc}\bigl(\R_+;L^{p_2}(\R^d)\bigr)$. Let $(\mu,\lambda)\in\R^2$ satisfy condition \eqref{lame}.

Then, for all $\,T>0$,  System \eqref{eq:lame} has a unique solution 
$h$ in $\cC([0,T];\dot B^{s_2}_{p_2,r_2})\cap L^{r_0}\bigl([0,T];L^{p_0}\bigr)$, with $\nabla h\,\in\,L^{r_1}\bigl([0,T];L^{p_1}\bigr)$ and
$\pa_th\,,\,\nabla^2h\,\in\,L^{r_2}\bigl([0,T];L^{p_2}\bigr)$.
Moreover, there exists a constant $C_0>0$ (depending just on $\mu$, $\lambda$, $d$, $p_0,$ $p_1,$ $p_2$ and $r_2$) such that the following estimate holds true:
\begin{multline} \label{est:max-reg_h}
\left\|h\right\|_{L^{\infty}_T(\dot B^{s_2}_{p_2,r_2})}\,+\,\left\|h\right\|_{L^{r_0}_T(L^{p_0})}\,+\,\left\|\nabla h\right\|_{L^{r_1}_T(L^{p_1})}\,+\,
\left\|\left(\pa_th,\nabla^2h\right)\right\|_{L^{r_2}_T(L^{p_2})}\\\,\leq\,C_0\,
\left(\left\|h_0\right\|_{\dot B^{s_2}_{p_2,r_2}}\,+\,\|f\|_{L^{r_2}_T(L^{p_2})}\right)\cdotp
\end{multline}
\end{prop}

\begin{proof}
Let us write down the Helmholtz decomposition of the vector field $h$: denoting by $\P$ the Leray projector onto the space of divergence-free
vector fields and by $\Q$ the projector onto the space of irrotational vector fields, we have $h\,=\,\P h\,+\,\Q h$.
Recall that  we have in Fourier variables
$$
\mc F\bigl(\Q h\bigr)(\xi)\,=\,\frac{1}{|\xi|^2}\,\bigl(\xi\cdot \what{h}(\xi)\bigr)\,\xi\,.
$$
Hence  $\P$ and $\Q$ are  linear combinations of composition of Riesz transforms and thus act continuously on $L^p$ for all $1<p<\infty.$
\medbreak
Now, applying those two operators to System \eqref{eq:lame}, we discover that $\P h$ and $\Q h$ satisfy the heat equations
$$
\bigl(\pa_t\,-\,\mu\,\Delta\bigr)\,\P h\,=\,\P f\quad\mbox{ and }\quad \bigl(\pa_t\,-\,\nu\,\Delta\bigr)\,\Q h\,=\,\Q f\,,
\quad\hbox{with }\ \nu:=\,\mu+\lambda.
$$
with initial data $\P h_0$ and $\Q h_0,$ respectively.  Therefore, denoting by  $\A$ the operator $\P$ or $\Q,$ and by $\alpha$ either $\mu$ or $\nu$,
Duhamel's formula gives us
\begin{equation} \label{eq:Duhamel}
\A h(t)\,=\,e^{\alpha t \Delta}\,\A h_0\,+\,\int^t_0e^{\alpha (t-s)\Delta}\,\A f(s)\,ds\,.
\end{equation}
For the term containing the initial datum, we apply  \eqref{est:nabla^2-h}, \eqref{est:ch}, \eqref{est:nabla-h} and \eqref{est:h}, while for the source term we apply Lemmas \ref{l:max-reg_2}, \ref{l:max-reg_1} and \ref{l:max-reg_0}.
Therefore, we conclude by continuity of operators $\P$ and $\Q$ over $L^p$, for $1<p<\infty$, and over $\dot B^{s}_{q,r}$, for all $s\in\R$ and all $(q,r)\in]1,\infty[^2$.
\end{proof}



\subsection{Tangential regularity} \label{ss:tang}

We establish here fundamental stationary estimates about propagation of \emph{striated} (or \emph{tangential}) \emph{regularity} in 
the $L^p$ setting.

Before starting the presentation, let us recall  that  the classical result on zero order pseudo-differential operators  ensures that, if $g\in L^\infty$, then $\Delta^{-1}\pa_i\pa_j g\in BMO$.
The main result of this paragraph states  that,
if $g$ has  suitable tangential regularity properties (similar to those
exhibited by J.-Y. Chemin in \cite{Ch1991}  for the  vortex patches problem),
then    $\Delta^{-1}\pa_i\pa_j g$ is  in $L^\infty$. 
Our starting point is an adaptation of  Lemma  5.1 in \cite{PZ}
enabling us  to handle operator $\nabla(\eta\Id-\Delta)^{-1}$, valid in any dimension
 and for non-zero divergence vector fields. 

Before stating that lemma, the proof of which is postponed in Appendix, let us introduce, for  $m\in\R$, the class $S^m$  of  symbols of order $m$, 
that is  the space of $\mc C^\infty(\R^d)$ functions $\sigma$  such that 
for all $\alpha\in\N^d,$ there exists   $C_\alpha>0$ satisfying for all $\xi\in\R^d$, 
$$\left|\p^\alpha \sigma(\xi)\right|\,\leq\,C_\alpha\,(1+|\xi|)^{m-|\alpha|}.$$
\begin{lem}\label{l:symb}
Let $1<p<\infty$. Consider a vector field $X$ in $\L^{\infty,p}$ and a function $g\in L^\infty$ such that $\pa_Xg\in L^p$.
Let $\sigma$ be a smooth Fourier multiplier in the class $S^{-1}.$ Then, for any fixed $0<s<1$, the following estimate holds true: 
$$
\left\|\p_X\s(D)g\right\|_{B^{s}_{p,\infty}}\,\leq\,C\,\bigl(\|\p_Xg\|_{L^p}\,+\,\|\na X\|_{L^p}\,\|g\|_{L^\infty}\bigr)\cdotp
$$
\end{lem}
Since the nonhomogeneous Besov space $B^s_{p,\infty}$ is embedded in $L^\infty$ whenever  $s>d/p,$ 
that   lemma implies the following fundamental result.
\begin{col}\label{c:symb}
Assume that  $d<p<\infty,$ and  consider a vector field $X$ in $\L^{\infty,p}$ and a function $L^\infty$ such that $\pa_Xg\in L^p$.
Then, for any   Fourier multiplier $\sigma$ in the class $S^{-1},$
there exists a constant $C>0$ such that 
$$\left\|\p_X\s(D)g\right\|_{L^\infty}\,\leq\,C\,\bigl(\left\|\p_Xg\right\|_{L^p}\,+\,\|\na X\|_{L^p}\,\|g\|_{L^\infty}\bigr)\cdotp$$
\end{col}

From the previous results, we immediately obtain the following fundamental stationary estimate.

\begin{prop} \label{p:tang}
 Fix $p\in\,]d,\infty[\,$ and an integer $m\geq d-1$, and take a non-degenerate family $\mc X=\left(X_\lambda\right)_{1\leq\lambda\leq m}$ of
vector-fields belonging to $\L^{\infty,p}$. Let $g\in L^\infty(\R^d)$ be such that $g\,\in\,\L^{p}_{\mc X}$.

Then for all $\eta>0$, one has the property $(\eta\Id-\Delta)^{-1}\nabla^2 g\,\in\,L^\infty(\R^d)$. Moreover, there exists a positive constant $C$ such that the following estimates hold true:
$$
\left\|\nabla^2(\eta\Id-\Delta)^{-1}g\right\|_{L^\infty}\,\leq\,C\Biggl(
\biggl(1+\frac{\tn{\mc X}_{L^\infty}^{4d-5}\,\tn{\nabla\mc X}_{L^p}}{\bigl(I(\mc X)\bigr)^{4d-4}}\biggr)\,\|g\|_{L^\infty}\,
+\,\frac{\tn{\mc X}_{L^\infty}^{4d-5}}{\bigl(I(\mc X)\bigr)^{4d-4}}\,\tn{\pa_{\mc X}g}_{L^p}\Biggr)\cdotp
$$
\end{prop}
\begin{proof} Fix some $\Lambda:=(\lambda_1,\cdots,\lambda_{d-1})\in \Lambda_{d-1}^m$, and consider
the set $U_\Lambda$ of those $x\in\R^d$ such that one has
$\Bigl(\stackrel{d-1}{\wedge}X_\Lambda\Bigr)(x)\geq \bigl(I(\mc X)\bigr)^{d-1}$.
By Lemma 3.2 of \cite{D1999}, there exists a family of functions $b_{ij}^{k\ell}$, where
$(i,j,k)\in\{1,\cdots,d\}^{3}$ and $\ell\in\{1\ldots d-1\}$,
which are homogeneous of degree $4d-5$ with respect to the coefficients of 
$\bigl(X_{\lambda_1}\ldots X_{\lambda_{d-1}}\bigr)$ and such that the following
identity holds true on $U_\Lambda$, for all $\xi\in \R^d$:
$$
\xi_i\,\xi_j\,=\,\frac{ \Bigl(\stackrel{d-1}{\wedge}X_\Lambda(x)\Bigr)^i\;\Bigl(\stackrel{d-1}{\wedge}X_\Lambda(x)\Bigr)^j}
{\Bigl|\stackrel{d-1}{\wedge}X_\Lambda(x)\Bigr|^2}\,|\xi|^2\,+\,\frac1{\Bigl|\stackrel{d-1}{\wedge}X_\Lambda(x)\Bigr|^4}\sum_{k,\ell}b_{ij}^{k\ell}\,\xi_k\,\left(X_{\lambda_\ell}(x)\cdot\xi\right)\,.
$$
Then, we multiply both sides by $(\eta+|\xi|^2)^{-1}\,\wh g(\xi)$ and take the inverse Fourier transform at $x$: 
$$\displaylines{\quad
(\eta\Id-\Delta)^{-1}\pa_i\pa_jg\,=\,\frac{ \Bigl(\stackrel{d-1}{\wedge}X_\Lambda\Bigr)^i\,\Bigl(\stackrel{d-1}{\wedge}X_\Lambda\Bigr)^j}{\Bigl|\stackrel{d-1}{\wedge}X_\Lambda\Bigr|^2}\,
\Delta(\eta\Id-\Delta)^{-1}g\hfill\cr\hfill\,+\,\frac1{\Bigl|\stackrel{d-1}{\wedge}X_\Lambda\Bigr|^4}\sum_{k,\ell} b_{ij}^{k\ell}\,\pa_{X_{\lambda_\ell}}\left(\pa_k(\eta\Id-\Delta)^{-1}g\right)\cdotp\quad}
$$
Hence, thanks also to Lemma \ref{l:D-1}, for all $\Lambda\in\Lambda^{d-1}$, we deduce the following bound on the set $U_\Lambda$:
\begin{multline} \label{est:tang-reg_U}
\left\|(\eta\Id-\Delta)^{-1}\pa_i\pa_j\,g\right\|_{L^\infty(U_\Lambda)}\leq\|g\|_{L^\infty}\\+\frac{C}{\bigl(I(\mc X)\bigr)^{4d-4}}\sum_{\ell}
\|X_\Lambda\|_{L^\infty}^{4d-5}\,\left\|\pa_{X_{\lambda_\ell}}\left(\nabla(\eta\Id-\Delta)^{-1}g\right)\right\|_{L^\infty}\,.
\end{multline}

In order to bound the last term in the right-hand side,  we apply Corollary \ref{c:symb}:
we get, for some constant $C>0$ also depending on $d$ and on $p$, the estimate
$$
\left\|\pa_{X_{\lambda_\ell}}\left(\nabla(\eta\Id-\Delta)^{-1}g\right)\right\|_{L^\infty}\,\leq\,C\,\left(\left\|\pa_{X_{\lambda_\ell}}g\right\|_{L^p}\,+\,
\|g\|_{L^\infty}\,\left\|\nabla X_{\lambda_\ell}\right\|_{L^p}\right)\cdotp
$$
Inserting this  bound into \eqref{est:tang-reg_U} immediately gives us the result.
\end{proof}

We conclude this part by presenting a new estimate concerning tangential regularity. This statement 
that will be proved in the appendix, turns out to be  of tremendous importance  to exhibit the  Lipschitz regularity
of the velocity field $u$, see Subsection \ref{ss:tangential} below.
\begin{prop} \label{p:tang-d_X}
Let the hypotheses of Proposition \ref{p:tang} be in force. 
Then   there exists a  constant $C>0$ such that, for all $\eta>0$, the following estimate holds true:
\begin{align*}
&\tn{\pa_{\mc X}\nabla^2(\eta\Id-\Delta)^{-1}g}_{L^p}\,\leq\,C\,\left(\tn{\nabla\mc X}_{L^p}
\biggl(1+\frac{\tn{\mc X}_{L^\infty}^{4d-5}\,\tn{\nabla\mc X}_{L^p}}{\bigl(I(\mc X)\bigr)^{4d-4}}\biggr)\,\|g\|_{L^\infty}\right. \\
&\qquad\qquad\qquad\qquad\left.+\,\biggl(1+\frac{\tn{\mc X}_{L^\infty}^{4d-5}\,\tn{\nabla\mc X}_{L^p}}{\bigl(I(\mc X)\bigr)^{4d-4}}\biggr)\,\tn{\pa_{\mc X}g}_{L^p}\,+\,
\frac{\tn{\mc X}_{L^\infty}^{4d-4}}{\bigl(I(\mc X)\bigr)^{4d-4}}\,\tn{\nabla\mc X}_{L^p}\,\|g\|_{L^\infty}\right)\,.
\end{align*}
\end{prop}


\section{An existence statement for almost critical data and only bounded density} \label{s:max-reg}

The goal of the present section is to prove  Theorem \ref{thm:existence}. After reformulating the original system \eqref{NS}, we establish \textsl{a priori} estimates on smooth solutions, then provide 
the reader with the construction of a family of approximate smooth solutions to our (new) system. 
As a last step, we   show the  convergence of the sequence to a true solution.

\subsection{Reformulation of the system}
In all that follows,  we assume for notational simplicity that $\nu=1.$ This is of course not restrictive, owing  to 
the change of unknowns
\begin{equation}\label{eq:change}
(\wt\rho,\wt u)(t,x)\,:=\,(\rho,u)(\nu t,\nu x)\,.
\end{equation}

First of all, we want to reformulate our system in terms of new unknowns, to which maximal regularity results of Subsection \ref{ss:max} apply.
As already explained at the beginning of the paper, in order to  handle the pressure term, 
it is convenient to introduce the auxiliary vector-field 
\begin{equation} \label{d:vbis}
 v\,:=\,-\nabla(\Id-\Delta)^{-1}P\,,
\end{equation}
and the  following  \emph{modified velocity field}:
\begin{equation} \label{d:w}
 w\,:=\,u\,-\,v\,=\,u\,+\nabla(\Id-\Delta)^{-1}P\,.
\end{equation}
For future uses, we  observe that 
\begin{equation} \label{eq:div-w}
\div u\,=\,\div w\,-\Delta(\Id-\Delta)^{-1}P\,.
\end{equation}


We want to reformulate  System \eqref{NS} in terms of the new unknowns $(\vrho,w)$, keeping in mind that  estimates for $v$  may be deduced from  those  for $\vrho$,
and that combining  with information on $w$ enables us to bound  the original velocity field $u$.
From  the first  equation of \eqref{NSbis}   and  relation \eqref{eq:div-w}, we immediately deduce that 
\begin{equation} \label{eq:vrho}
 \pa_t\vrho\,+\,u\cdot\nabla\vrho \,=\, -\,\rho\,\div w\,+\,\rho\Delta(\Id-\Delta)^{-1}P\,,
\end{equation}
with $\rho\,:=1+\vrho$ and $u\,:=\,w\,-\,\nabla(\Id-\Delta)^{-1}P.$ 
\medbreak
Regarding $w,$ we see from  the second equation of \eqref{NSbis} and  \eqref{d:w} that
\beqo
\rho\,\pa_t w\,-\,\mu\,\Delta w\,-\,\la\,\na \dive w\,=\,-\,(\Id-\Delta)^{-1}\nabla P-\rho u\cdot \na u\,-\,\rho(\Id-\D)^{-1}\na\p_tP.
\eeqo
Using once again the mass equation in \eqref{NS}, we find
$$
\pa_tP\,+\,\div\bigl(P\,u\bigr)\,=\,g(\rho)\div u\,\with g(\rho):=P(\rho)-\rho P'(\rho),
$$
so that the equation for $w$ can be recast in
\begin{equation} \label{eq:w}
\rho\pa_tw\,+\,\mc Lw\,=\,-\,\rho F\,,
\end{equation}
with
\begin{multline}
F\,=\,\frac1\rho(\Id-\Delta)^{-1}\nabla P\,+\,w\cdot\na w\,+\,w\cdot \na v\, +\,v\cdot\na w\,+\,v\cdot\na v \label{d:F} \\
+\,(\Id-\Delta)^{-1}\nabla\Bigl(g(\rho)\,\dive u\,-\,\dive\bigl(P\,u\bigr)\Bigr)\cdotp
\end{multline}

The existence part of Theorem \ref{thm:existence} will be a consequence of  the following statement. 
\begin{prop} \label{p:exist_w-r}
Let $\vrho_0\,\in\,L^{p}\cap L^\infty(\R^d)$ and $w_0\,\in\,\dot B^{2-2/r}_{p,r}$, with
\begin{equation}\label{eq:pr}
d\,<\,p\,<\,\infty\qquad\hbox{ and }\qquad 1\,<\,r\,<\,\frac{2p}{2p-d}\cdotp
\end{equation}
Define $r_0$ and $r_1$ by 
\begin{equation}\label{eq:r0r1}\frac1{r_0}\,:=\,\frac1r\,+\,\frac d{2p}-1\andf \frac1{r_1}\,:=\,\frac1r\,-\,\frac12\cdotp\end{equation}
There exists  some small enough constant $\veps>0$ depending only on $d$, $p$ and $r$, 
such that, if in addition $\vrho_0$ fulfills \eqref{hyp:small-vrho},
then there exist a time $T>0$ and a weak solution $(\vrho,w)$ to system \eqref{eq:vrho}-\eqref{eq:w} on $[0,T[\,\times\R^d$ such that $\vrho\in\cC([0,T];L^{p})$ and
\begin{equation}\label{est:small-vrho} \|\vrho\|_{L^\infty_T(L^\infty)}\,\leq\,4\,\veps\,,\end{equation}
and $w\,\in\,\cC\bigl([0,T];\dot B^{2-2/r}_{p,r}\bigr)\cap\,L^{r_0}\bigl([0,T]; L^{\infty}(\R^d)\bigr)$ with $\nabla w\,\in\,L^{r_1}\bigl([0,T]; L^{p}(\R^d)\bigr)$ and
$\pa_tw\,,\,\nabla^2w\,\in\,L^r\bigl([0,T];L^p(\R^d)\bigr)$. \\
If in addition to the above hypotheses, we have  $\inf P'>0$ on $[1-4\veps,1+4\veps]$ and 
$\vrho_0, u_0  \,\in\,L^2(\R^d)$, where $u_0\,:=\,w_0+v_0$ and $v_0$ is defined as in the statement of Theorem \ref{thm:existence}, then 
 $u\,:=\,w\,-\,\nabla(\Id-\Delta)^{-1}P$ fulfills  $u\,\in\,L^\infty_T(L^2)\,\cap\,L^2_T(H^1)$  and 
 the energy equality \eqref{est:en-est} holds true.
\end{prop}

\subsection{\textsl{A priori} bounds for smooth solutions} \label{ss:a-priori}

We start by establishing \textsl{a priori} estimates for smooth solutions to the new system \eqref{eq:vrho}-\eqref{eq:w}. 
 Our goal  is to ``close the estimates''  in the space
$$
E_T\,:=\,\Bigl\{\bigl(\vrho,w\bigr)\in  L^\infty_T(L^{p}\cap L^\infty)\times\bigl(\mc C([0,T];\dot B^{2-2/r}_{p,r})\cap L^{r_0}_T(L^{\infty})\bigr)\;\bigl|\quad
\na w \in L^{r_1}_T(L^{p})\,,\;\nabla^2w\in L^{r_2}_T(L^{p})\Bigr\}
$$
for some small enough  $T>0.$ 
\medbreak
We denote
\begin{equation}\label{eq:NT}
\mc N(T)\,:=\,\|\vrho\|_{L^\infty_T(L^{p}\cap L^\infty)}\,+\,\left\|w\right\|_{L^\infty_T(\dot B^{2-2/r}_{p,r})\cap L^{r_0}_T(L^{\infty})} 
\,+\,\left\|\nabla w\right\|_{L^{r_1}_T(L^{p})}
\,+\,\left\|\bigl(\pa_tw\,,\,\nabla^2w\bigr)\right\|_{L^{r}_T(L^{p})}\,.
\end{equation}

As it fulfills a  transport equation,  it is easy to propagate any Lebesgue norm $L^q$ for $\vrho$
once  we know that $\div u$ is in $L^1_T(L^\infty)$ and  that the right-hand side of \eqref{eq:vrho} is   in $L^1_T(L^q).$
Given the expected properties  on $w,$ this will give us the constraint   $q\geq p$.
As for equation \eqref{eq:w}, we want to apply the 
maximal regularity estimates given by Proposition \ref{p:max-reg}.

\subsubsection{Bounds for the density} \label{sss:ub-vrho}
Throughout, we fix some $\varepsilon>0$ and  constant $C>0$ so that (recall that $P(1)=0$)
\begin{equation}\label{eq:smallrho}
|P(z)|\,\leq\,C\,|z-1|\qquad\mbox{ for all }\; z\in[1-4\veps,1+4\veps]\,.
\end{equation}

As a first step, let us establish estimates for the density term. 
Let us take some $q\in[p,\infty[.$ By multiplying equation  \eqref{eq:vrho} by $|\vrho|^{q-2}\,\vrho$ and integrating in space, we easily get
$$
\frac{1}{q}\,\frac{d}{dt}\|\vrho\|^q_{L^q}\,-\,\frac{1}{q}\int|\vrho|^q\,\div u\,+\,\int|\vrho|^q\,\div w\,
+\,\int\vrho\,|\vrho|^{q-2}\div w\,=\,\int  (\vrho+1)\vrho|\vrho|^{q-2}\Delta(\Id-\Delta)^{-1}P\,.
$$
By \eqref{eq:div-w}, one can rewrite the previous relation as
$$\frac{1}{q}\,\frac{d}{dt}\|\vrho\|^q_{L^q}\,+\,\Bigl(1\,-\,\frac{1}{q}\Bigr)\int|\vrho|^q\bigl(\div w
\,-\,\Delta(\Id-\Delta)^{-1}P\bigr)\,=\,\int\vrho|\vrho|^{q-2}\bigl(\Delta(\Id-\Delta)^{-1}P-\div w\bigr),
$$
which immediately implies 
$$\displaylines{\quad
\|\vrho(t)\|_{L^q}\,\leq\, \|\vrho_0\|_{L^q}\,
+\,\Bigl(1\,-\,\frac{1}{q}\Bigr)\int^t_0\|\vrho\|_{L^q}\|\Delta(\Id-\Delta)^{-1}P-\div w\|_{L^\infty}\,d\tau
\hfill\cr\hfill+\int_0^t\|\Delta(\Id-\Delta)^{-1}P-\div w\|_{L^q}\,d\tau.\quad}
$$
Of course, passing to the limit $q\to\infty,$ we see that the same inequality holds true for $q=\infty.$
Then, using Lemma \ref{l:D-1} and the fact that, under \eqref{eq:smallrho}, we have 
\begin{equation}\label{eq:Pcont}
\|P\|_{L^r}\leq C\|\vrho\|_{L^r}\qquad\hbox{ for all }\; r\in[1,\infty]\,,
\end{equation}
an application of Gronwall lemma implies that, for all $q\in[p,\infty]$,  there holds
\begin{equation} \label{est:vrho_inf}
\|\vrho(t)\|_{L^q}\,\leq\,e^{Ct+\int^t_0\|\div w\|_{L^\infty}\,d\tau}\,\biggl(\|\vrho_0\|_{L^q}\,+\,\int^t_0\|\div w\|_{L^q}\,d\tau\biggr)\cdotp
\end{equation} Define now the time $T>0$ as
\begin{equation} \label{d:T}
T\,:=\sup\left\{t>0\;|\;Ct+\int^t_0\|\div w\|_{L^\infty}\,d\tau\,\leq\,\log2
\quad \hbox{and}\quad \int_0^t\|\div w\|_{L^\infty}\leq\veps\right\}\,;
\end{equation}
then, on $[0,T]$ one has, owing to \eqref{est:vrho_inf},
$$
\|\vrho(t)\|_{L^\infty}\,\leq\,2\,\|\vrho_0\|_{L^\infty}\,+2\,\veps\,.
$$
Hence, if we take the initial density satisfying \eqref{hyp:small-vrho}, for $\veps$ fixed in \eqref{eq:smallrho} above, then we get \eqref{est:small-vrho}. 



\subsubsection{Bounds for $w$} \label{sss:ub-w} Throughout we  fix the time $T>0$ as defined in \eqref{d:T}
and assume that \eqref{est:small-vrho} is fulfilled.
Then, applying Proposition \ref{p:max-reg}  to equation \eqref{eq:w} with $(p_0,p_1,p_2)=(\infty,p,p),$ $r_2=r$ and $(r_0,r_1)$  according 
to  \eqref{eq:r0r1},  we get, treating  the term $\vrho\,\pa_tw$ as a perturbation, 
$$
\displaylines{\quad
\left\|w\right\|_{L^\infty_T(\dot B^{2-2/r}_{p,r})\cap L^{r_0}_T(L^{\infty})}\,+\,\left\|\nabla w\right\|_{L^{r_1}_T(L^{p})}\,+\,
\left\|\bigl(\pa_tw\,,\,\nabla^2w\bigr)\right\|_{L^{r}_T(L^{p})}\,\leq \hfill\cr\hfill
\qquad\qquad\qquad
\leq\,C_0\,\left(\left\|w_0\right\|_{\dot B^{2-2/r}_{p,r}}\,+\,\|\vrho\|_{L^\infty_T(L^\infty)}\,\left\|\pa_tw\right\|_{L^{r}_T(L^{p})}\,+\,\|F\|_{L^{r}_T(L^{p})}\right)\,,\quad}
$$
for a suitable constant $C_0>0.$ 
Now, assuming that $\veps$ in \eqref{hyp:small-vrho} has been fixed so small that
\begin{equation}\label{eq:small-veps}
8\,C_0\,\veps\,\leq\,1\,,
\end{equation}
thanks to \eqref{est:small-vrho} we gather the estimate
\begin{multline}\label{eq:estw}
\left\|w\right\|_{L^\infty_T(\dot B^{2-2/r}_{p,r})\cap L^{r_0}_T(L^{\infty})}\,+\,\left\|\nabla w\right\|_{L^{r_1}_T(L^{p})}
\,+\,\left\|\bigl(\pa_tw\,,\,\nabla^2w\bigr)\right\|_{L^{r_2}_T(L^{p})}\\\,\leq\,2\,C_0\,
\left(\left\|w_0\right\|_{\dot B^{2-2/r}_{p,r}}\,+\,\|F\|_{L^{r}_T(L^{p})}\right)\cdotp
\end{multline}

Our next goal is to bound $F$, defined by \eqref{d:F}.
First of all, as  operator $\nabla(\Id-\Delta)^{-1}$ maps continuously $L^q$ into $W^{1,q}$ for any $1<q<\infty$, one deduces  that
\begin{equation}
 \label{est:v-L^p}
 \|v\|_{W^{1,q}}\,\leq\, C \|\vrho\|_{L^q}\,,
\end{equation}
where $v$ is the vector-field defined in \eqref{d:vbis}. Hence, the first term in \eqref{d:F} can be bounded, thanks to  \eqref{eq:Pcont}, in the following way:
\begin{equation}\label{eq:w1}
\|(\Id-\Delta)^{-1}\nabla P\|_{L^{r}_T(L^{p})}\,
\leq\, C\, T^{1/r}\, \|\vrho\|_{L^\infty_T(L^{p})}\,.
\end{equation}

Next, we estimate the transport terms of $F$  by means of  H\"older inequality, using that
$$
\frac 1{r}-\frac1{r_0}-\frac1{r_1}=\frac32-\frac1{r}-\frac d{2p}=\frac12-\frac1{r_0}>0.$$
We get the following inequality:
\begin{multline}
\|w\cdot\na w\,+\,w\cdot \na v\, +\,v\cdot\na w \,+\,v\cdot\na v\|_{L^{r}_T(L^{p})}\,\leq \label{est:F_transp} \\
\leq\,C\,T^{\frac1{2}-\frac1{r_0}}\left(\|v\|_{L_T^{r_0}(L^{\infty})}\,+\,\|w\|_{L_T^{r_0}(L^{\infty})}\right)
\left(\|\nabla v\|_{L_T^{r_1}(L^{p})}\,+\,\|\nabla w\|_{L_T^{r_1}(L^{p})}\right)\cdotp
\end{multline}
  Hence, using the definition of $\cN(T)$ and \eqref{est:v-L^p}, inequality \eqref{est:F_transp} becomes
  \begin{equation}\label{eq:vw}
  \|w\cdot\na w\,+\,w\cdot \na v\, +\,v\cdot\na w \,+\,v\cdot\na v\|_{L^{r}_T(L^{p})}\,\leq\,
  C\, T^{\frac1{2}-\frac1{r_0}}(1+T^{\frac1{r_0}})(1+T^{\frac1{r_1}})\cN^2(T)\,.
  \end{equation}
  
Let us now consider the term $\nabla(\Id-\Delta)^{-1}\div(Pu)$  occurring in the definition of $F$. 
{}From  Corollary \ref{c:Delta-Id} and \eqref{est:small-vrho}, combined with the continuity of the function $P$, we deduce that
$$\begin{aligned}
\left\|\nabla(\Id-\Delta)^{-1}\div\bigl(P\,u\bigr)\right\|_{L^{r}_T(L^{p})}\,&\leq\, C\,\left\|P\,u\right\|_{L^{r}_T(L^{p})} \\
&\leq\,C\,T^{\frac1{r}-\frac1{r_0}-\frac1{r_1}}\,\|P\|_{L^{r_1}_T(L^{p})}\,\bigl(\|v\|_{L^{r_0}_T(L^{\infty})}+\|w\|_{L^{r_0}_T(L^{\infty})}\bigr) \\
&\leq\, C\,T^{\frac1{r}-\frac1{r_0}}\,\|\vrho\|_{L^\infty_T(L^{p})}\,\left(T^{\frac1{r_0}}\|\vrho\|_{L^{\infty}_T(L^\infty)}+\|w\|_{L^{r_0}_T(L^{\infty})}\right)\cdotp
\end{aligned}$$
Hence, using the definition of  $\cN(T),$ we get
\begin{equation}\label{eq:divPu}
\left\|\nabla(\Id-\Delta)^{-1}\div\bigl(P\,u\bigr)\right\|_{L^{r}_T(L^{p})}
\,\leq\, C\,T^{\frac1{r}-\frac1{r_0}}\, (1+T^{\frac1{r_0}})\,\cN^2(T)\,.
\end{equation}

To handle  the last term of $F,$ we write that
$$\nabla(\Id-\Delta)^{-1}\bigl(g(\rho)\,\dive u\bigr)=g(1)\,\nabla(\Id-\Delta)^{-1}\div u\,+\,\nabla(\Id-\Delta)^{-1}\Bigl(\bigl(g(\rho)\,-\,g(1)\bigr)\,\dive u\Bigr)\cdotp
$$
To bound the first term, we use that, thanks to \eqref{eq:div-w},  
$$
\nabla(\Id-\Delta)^{-1}\div u =  \nabla(\Id-\Delta)^{-1}\div w
+\nabla(\Id-\Delta)^{-2}\Delta P\,.
$$
Because both $\nabla(\Id-\Delta)^{-1}$ and  $\nabla(\Id-\Delta)^{-2}\Delta$ map $L^{p}$ to itself, we get
$$
\left\|\nabla(\Id-\Delta^{-1})\div u\right\|_{L^{r}_T(L^{p})}\,\leq\,C\,
\left(T^{\frac12}\,\|\nabla w\|_{L_T^{r_1}(L^{p})}\, +\, T^{\frac1{r}}\,\|P\|_{L_T^{\infty}(L^{p})}\right)\,\leq\,C\,\left(T^{\frac12}\,+\,T^{\frac1{r}}\right)\, \cN(T)\,.
$$
Similarly, we have
$$\begin{aligned}
\left\|\nabla(\Id-\Delta)^{-1}\Bigl(\bigl(g(\rho)\,-\,g(1)\bigr)\,\dive u\Bigr)\right\|_{L^{r}_T(L^{p})}&\leq C\,\left\|\bigl(g(\rho)\,-\,g(1)\bigr)\,\dive u\right\|_{L^{r}_T(L^{p})}\\
&\leq\,C\, T^{\frac1{r}-\frac1{r_0}-\frac1{r_1}}\,\|\vrho\|_{L^{r_0}_T(L^{\infty})}\,\|\dive u\|_{L^{r_1}_T(L^{p})}\,;
\end{aligned}$$
so, keeping in mind \eqref{eq:div-w}, one can conclude that 
\begin{equation}\label{eq:lastterm}
\left\|\nabla(\Id-\Delta)^{-1}\Bigl(\bigl(g(\rho)\,-\,g(1)\bigr)\,\dive u\Bigr)\right\|_{L^{r}_T(L^{p})}
\,\leq\, C\, T^{\frac1{2}}\,(1+T^{\frac1{r_1}})\, \cN^2(T)\,.
\end{equation}


In the end, plugging the inequalities \eqref{est:vrho_inf} and \eqref{eq:w1} to \eqref{eq:lastterm} 
in \eqref{eq:estw},  
whenever relation \eqref{est:small-vrho} is fulfilled, we get
$$
\cN(T)\,\leq\, C\,\left(\|\vrho_0\|_{L^{p}\cap L^\infty}\,+\,\|w_0\|_{\dot B^{2-2/r}_{p,r}}\,+\,
\bigl(T^{\frac12}+T^{\frac1{r}}\bigr)\, \cN(T)\, +\, \bigl(T^{\frac12-\frac1{r_0}}+T^{\frac1{r}}\bigr)\, \cN^2(T)\right),
$$
for some constant $C$ depending only on the pressure function and on the regularity parameters.
{From} that inequality and a standard bootstrap argument, one may conclude that
there exists a time  $T>0$, depending only on the norm of the initial data, such that 
\begin{equation}\label{est:N}
\cN(T)\,\leq\, 2\,C\, \left(\|\vrho_0\|_{L^{p}\cap L^\infty}\,+\,\|w_0\|_{\dot B^{2-2/r}_{p,r}}\right)\cdotp\end{equation}


\subsubsection{Classical energy estimates} \label{sss:en-est}
We establish here energy estimates for solutions to system \eqref{NS}, under the additional assumption that $u_0\in L^2$.
The computations being quite standard, we only sketch the arguments, and refer the reader to  e.g. Chapter 5 of \cite{PLL2} for the details:  we take the $L^2$ scalar product of the momentum equation in \eqref{NS} with $u$, integrate by parts and  make use of the mass equation. 
Defining $\Pi$ as in the statement of Theorem \ref{thm:existence},  we end up with the relation
$$
\frac{1}{2}\,\frac{d}{dt}\int\rho|u|^2\,dx\,+\,\frac{d}{dt}\int\Pi(\rho)\,dx\,+\,\mu\int|\nabla u|^2\,dx\,+\,\lambda\int|\div u|^2\,dx\,=\,0\,.
$$
The previous relation, after  integration in time, leads to the classical energy balance \eqref{est:en-est}. 

Now, keeping in mind   the smallness assumption 
\eqref{hyp:small-vrho}, we gather that if $\inf P'>0$ on $[1-4\veps,1+4\veps]$ then 
 it holds on $[0,T]$:
$$
C^{-1} \|\vrho(t)\|_{L^2}^2  \leq \left\|\Pi\bigl(\rho(t)\bigr)\right\|_{L^1}\,\leq\,C\,\|\vrho(t)\|_{L^2}^2\,.$$
 
Hence, by the hypotheses on the initial data, we get that the right-hand side of \eqref{est:en-est} is finite, and then, for all $t\in[0,T]$, one has that
$\sqrt{\rho}\,u$ belongs to $L^\infty_t(L^2)$, $\vrho\,\in\,L^\infty_t(L^2)$ and $\nabla u\,\in\,L^2_t(L^2)$.

To make a long story short, one can eventually assert that  for all $t\in[0,T],$ we have
\begin{equation} \label{est:classic}
\|u(t)\|^2_{L^2}\,+\,\int^t_0\|\nabla u(\tau)\|^2_{L^2}\,d\tau\,+\,\|\vrho(t)\|_{L^2}^2\,\leq\,C\,,
\end{equation}
for some  $C>0$ just depending on  the initial energy, on $P$ and on $\varepsilon.$


\subsection{The proof of existence} \label{ss:existence}

In this subsection, we derive, from the estimates of the previous part, the existence of a weak solution to system \eqref{eq:vrho}-\eqref{eq:w}.

We start by smoothing out the initial data $(\vrho_0,u_0)$ by convolution with a family of nonnegative mollifiers:
$$
\vrho_{0}^n\,:=\,\chi^n\,*\,\vrho_0\andf
u_0^n\,:=\,\chi^n\,*\,u_0\,. $$
Then $\vrho^n_0$ still satisfies \eqref{hyp:small-vrho},
and both $\vrho^n_0$ and $u^n_0$ belong to all Sobolev spaces
$W^{k,p}$, with $k\in\N$.  Note that one can in addition multiply the regularized data by a family of cut-off
functions, to have $\vrho_0^n$ and $u_0^n$ in $L^2$, which enables us to apply Theorem  A of \cite{Mucha01}\footnote{Actually,  \cite{Mucha01} concentrates on the $\R^3$ case
but very small changes allow to get a similar result in $\R^d$ provided $p>d$.
One can alternately use  \cite{D01} that proves 
local-in-time existence in Besov spaces with sub-critical regularity.}. We get a sequence of solutions 
 $(\vrho^n,u^n)$ on $[0,T^n]$ (with $T^n>0$) to \eqref{NS}, supplemented  with initial data $(1+\vrho_0^n,u_0^n)$, 
 fulfilling \eqref{hyp:small-vrho}, the energy balance \eqref{est:en-est}, and the following properties:
 $$
 \vrho^n\in\cC([0,T^n];W^{1,p})\,,\qquad
 u^n\in\cC([0,T^n]; L^2)\,,\qquad
  \pa_tw^n,\nabla^2w^n\in L^{r}([0,T^n];L^{p})\,.$$
Furthermore, by taking advantage of Inequality \eqref{est:N}, one can 
 exhibit some $T>0$, depending only on norms of $(\vrho_0,u_0)$,
such that  $T^n\geq T$ for all $n\in\N$,  and eventually get, for some constant $C$
depending only on $p$, $r$, $d$ and $\veps$, the bound
$$\displaylines{\quad
\|\vrho^n\|_{L^\infty_T(L^{p}\cap L^\infty)}\,+\,\|w^n\|_{L^\infty_T(\dot B^{2-2/r}_{p,r})\cap L^{r_0}_T(L^{\infty})}\,+\,\left\|\nabla w^n\right\|_{L^{r_1}_T(L^{p})}\hfill\cr\hfill
\,+\,\left\|\bigl(\pa_tw^n\,,\,\nabla^2w^n\bigr)\right\|_{L^{r}_T(L^{p})}\leq  C \Bigl(\|\vrho_0\|_{L^{p}\cap L^\infty}+\|w_0\|_{\dot B^{2-2/r}_{p,r}}\Bigr)\,.\quad}
 $$
 
The previous inequality ensures the weak-$\star$ convergence (up to an extraction of a subsequence) of $(\vrho^n,w^n)$ to some $(\vrho,w)$ in the space $E_T.$

Strong convergence properties are still needed, in order to pass to the limit in the weak formulation of the equations, and show that $(\vrho,w)$ is indeed a solution of system
\eqref{eq:vrho}-\eqref{eq:w}.
In order to glean strong compactness, it suffices to use the fact that the above
uniform bound also provides a control on the first order time derivatives in sufficiently negative
Sobolev spaces through the equation fulfilled
by $(\vrho^n,w^n).$   Then one can combine with Ascoli theorem and interpolation, 
to get strong convergence, which  turns out to be enough to pass to the limit in the equation satisfied by $w.$
In order to justify that $\vrho$ satisfies the mass equation, one can 
repeat the arguments of \cite{HPZ} (see also \cite{PLL2}) which in particular imply that $\vrho^n\to \vrho$ 
(up to subsequence) in $\cC([0,T];L^q)$ for all $p\leq q<\infty.$ The details are left to the reader. 

Finally, once it is known that $(\vrho,w)$ satisfy the desired equation, one can recover time continuity 
for $w$ by taking advantage of  Proposition \ref{p:max-reg}.  
\smallbreak
To prove that  the energy balance is fulfilled  in the case where $\inf P'>0$ on $[1-4\veps,1+4\veps]$
and $\vrho_0,u_0\in L^2,$ 
we just have to observe that it is satisfied by $(\rho^n,u^n)$ (with the regularized data) for all $n\in\N$
and that  having  $\vrho_0\in L^2$ guarantees that $\vrho^n\to\vrho$ in $\cC([0,T];L^2).$
This implies that $v^n\to v$ in $\cC([0,T];H^1).$
Furthermore, the compactness properties of $(w^n)$ that have been pointed out just above
ensure that $w^n\to w$ in $L^2([0,T];H^1_{loc})\cap L^\infty([0,T];L^2_{loc}).$ 
Finally, since $$\displaylines{\frac12\left\|\sqrt{\rho(t)}\,u(t)\right\|^2_{L^2}\,+\,\left\|\Pi\bigl(\rho(t)\bigr)\right\|_{L^1}\,
+\,\mu\,\|\nabla u\|^2_{L^2_t(L^2)}\,+\,\lambda\,\|\div u\|^2_{L^2_t(L^2)}\hfill\cr\hfill=
\lim_{R\to\infty}\biggl(\frac12\int_{B(0,R)}\left(\rho(t,x)|u(t,x)|^2\,+\,\Pi\bigl(\rho(t,x)\bigr)\right)\,dx
+\int_0^t\int_{B(0,R)}\left(\mu\,|\nabla u|^2\,+\,\lambda\,(\div u)^2\right)\,dx\,d\tau\biggr),}$$
and because the aforementioned properties of convergence enable us to pass to the limit
in the right-hand side for any $R>0,$ we get the desired energy balance. 
\medbreak
Proposition \ref{p:exist_w-r} is thus completely proven, and so is  
Theorem \ref{thm:existence} (apart from uniqueness if $d=1$ that  will be discussed at the beginning
of the next section).


\section{Tangential regularity and uniqueness} \label{s:tang}

The main goal of this section is to prove uniqueness of solutions to \eqref{NS} in the previous functional framework.
According to the pioneering work by D. Hoff in \cite{Hoff3} or  to the recent paper \cite{D-Fourier} by the first author, the condition $\nabla u\in L^1_T(L^\infty)$ seems to be the minimal
requirement in order to get uniqueness.  
Recall that (still assuming that $\nu=1$ for notational simplicity) 
\begin{equation} \label{eq:du}
\nabla u=\nabla w+\nabla^2(\Id-\D)^{-1}P(\rho)\,.
\end{equation}
By Proposition \ref{p:exist_w-r} and Sobolev embeddings, we immediately get that $\nabla w$ is in  $L^1_T(L^\infty).$    So is the last term  of \eqref{eq:du} if the space dimension is $1$   since 
 $P(\rho)$ is bounded and $\pa_{xx}^2(\Id-\pa_{xx}^2)^{-1}$ maps $L^\infty$ to $L^\infty.$ 
 If $d\geq2$ however then the property that $P$ is bounded ensures only (by Calder\'on-Zygmund theory) that $\nabla^2(\Id-\D)^{-1} P$ is  in $BMO$.
Having Proposition \ref{p:tang} in mind, this prompts us to make an additional tangential regularity type assumption so as to guarantee that,
indeed, $\nabla^2(\Id-\D)^{-1}P$ belongs to $L^\infty$. 
\medbreak
In the rest of this section, we thus assume  the following tangential regularity hypothesis: for $p\in\,]d,\infty[\,$,
there exists a non-degenerate family $\mc X_0\,=\,\bigl(X_{0,\la}\bigr)_{1\leq\la\leq m}$ of vector-fields in $\L^{\infty,p}$ such that
the initial density $\rho_0$ belongs to the space $\L^p_{\mc X_0}$ (see Definition \ref{d:stri} above).


\subsection{Propagation of tangential regularity} \label{ss:tangential}

In this subsection we establish a priori estimates  for striated regularity of the density 
of a smooth enough solution  $(\vrho,u)$  to  system \eqref{eq:vrho}-\eqref{eq:w}.
From those bounds, we will infer a control on the Lipschitz norm of  $u$.
Throughout that section, we shall use the notation
$$
U(t)\,:=\,\int^t_0\|\nabla u(\tau)\|_{L^\infty}\,d\tau\,.
$$

\subsubsection{Bounds for the tangential vector fields}

Let us generically denote by $X_0$ one of the vector-fields of the family $\mc X_0$. 
It is well known that the evolution of $X_0$ along the velocity flow  is the solution to the transport equation
\begin{equation} \label{eq:X}
 \left\{\begin{array}{l}
         \left(\pa_t\,+\,u\cdot\nabla\right)X\,=\,\pa_Xu \\[1ex]
         X_{|t=0}\,=\,X_0\,,
        \end{array}\right.
\end{equation}
where the notation $\pa_Xu$ has been introduced in \eqref{eq:d_Yf}.
\smallbreak
For all $t\geq0$, we then define the family $\mc X(t)\,:=\,\bigl(X_\la(t)\bigr)_{1\leq\la\leq m}$, where $X_\la$  stands for the  solution to system \eqref{eq:X}, 
supplemented with  initial datum $X_{0,\la}$.

Our goal, now, is to establish some bounds on each vector-field $X(t)$ of the family $\mc X(t)$. They are based on classical estimates for transport equations 
in the spirit  of those of Paragraph \ref{sss:ub-vrho}.

First of all, the standard $L^\infty$ estimate for equation \eqref{eq:X} leads us to the inequality
\begin{equation} \label{est:X_L^p}
\|X(t)\|_{L^\infty}\,\leq\,\|X_0\|_{L^\infty}\;e^{U(t)}\;. 
\end{equation}

Next, arguing exactly as in Proposition 4.1 of \cite{D1999}
(one needs to pass through the flow associated to $u$) yields 
\begin{equation} \label{est:I}
I\bigl(\mc X(t)\bigr)\,\geq\,I(\mc X_0)\,e^{-U(t)}\,,
\end{equation}
which ensures  that the family $\mc X(t)$ remains non-degenerate whenever $U(t)$ stays bounded.
\medbreak
Finally, differentiating equation \eqref{eq:X} with respect to the space variable $x_j,$ we get
$$\pa_t\pa_jX\,+\,u\cdot\nabla\pa_jX\,=\,-\,\pa_ju\cdot\nabla X\,+\,\pa_j\pa_Xu\,,$$
which leads to 
\begin{equation} \label{est:dt_dX}
\|\nabla X(t)\|_{L^p}\,\leq\,\|\nabla X(0)\|_{L^p}\,+\int_0^t\bigl(C\|\nabla u\|_{L^\infty}\|\nabla X\|_{L^p}\,+\,\|\nabla\pa_X u\|_{L^p}\bigr)\,d\tau.
\end{equation}
Observe that, for all $1\leq j\leq d$, we have the relation
$$
\pa_j\pa_Xu\,=\,\pa_jX\cdot\nabla u\,+\,\pa_X\pa_ju\,,
$$
where the former term is easily bounded in $L^p$ by the quantity $\|\nabla X\|_{L^p}\,\|\nabla u\|_{L^\infty}$. 
Then, taking advantage of Gronwall inequality,  we  get
\begin{equation} \label{est:dX_2}
\left\|\nabla X(t)\right\|_{L^p}\,\leq\,e^{CU(t)}\left(\|\nabla X_0\|_{L^p}\,+\,\int^t_0e^{-CU(\tau)}\|\pa_X\nabla u\|_{L^p}\,d\tau\right)\cdotp
\end{equation}

\subsubsection{Propagation of striated regularity for the density} \label{sss:stri-rho}

We now show propagation of tangential regularity for the density function.  
To begin with, we recast the first equation of \eqref{NS} in the form
\begin{align}
\pa_t\rho\,+\,u\cdot\nabla\rho\,&=\,-\,\rho\,\div u\,. \label{eq:rho-transp}
\end{align}
Next, we multiply the previous relation by $X$: by virtue of \eqref{eq:X}, we find
$$
\pa_t\bigl(\rho\,X\bigr)\,+\,u\cdot\nabla\bigl(\rho\,X\bigr)\,+\,\rho\,X\,\div u\,
-\,\rho\,\pa_Xu\,=\,0\,.
$$
Taking the divergence of the obtained relation, straightforward computations lead  to the equation
$$
\pa_t\div\bigl(\rho\,X\bigr)\,+\,u\cdot\nabla\div\bigl(\rho\,X\bigr)\,=\,-\,\div u\;\div\bigl(\rho\,X\bigr)\,.
$$
From it, repeating the computations of Paragraph \ref{sss:ub-vrho}, we deduce that 
\begin{equation} \label{est:rho-tang}
\left\|\div\bigl(\rho(t)\,X(t)\bigr)\right\|_{L^p}\,\leq\,e^{CU(t)}\,\left\|\div\bigl(\rho_0\,X_0\bigr)\right\|_{L^p}\,.
\end{equation}

Thanks to the previous estimate and to the relation
$$
\div\bigl(P(\rho)\,X\bigr)\,=\,\bigl(P(\rho)\,-\,\rho\,P'(\rho)\bigr)\,\div X\,+\,P'(\rho)\,\div\bigl(\rho\,X\bigr)\,,
$$
one easily gathers the propagation of tangential regularity also for the pressure term:
$$\begin{aligned}
\left\|\div\bigl(P(\rho)\,X\bigr)\right\|_{L^p}&\,\leq\,\|P(\rho)\,-\,\rho P'(\rho)\|_{L^\infty}\,\|\div X\|_{L^p}\,+\,
\|P'(\rho)\|_{L^\infty}\,\left\|\div\bigl(\rho\,X\bigr)\right\|_{L^p}  \\
&\,\leq\,C\,\bigl(\|\nabla X\|_{L^{p}}\,+\,\left\|\div\bigl(\rho\,X\bigr)\right\|_{L^p}\bigr)\cdotp
\end{aligned}$$
where, in writing the last inequality, we have used \eqref{eq:smallrho} and \eqref{est:small-vrho}.
For future use, we also notice that, by the previous estimate and \eqref{eq:d_Yf}, we have
\begin{equation} \label{est:d_XP}
\left\|\pa_XP(\rho)\right\|_{L^p}\,\leq\,C\,\bigl(\|\nabla X\|_{L^{p}}\,+\,\left\|\div\bigl(\rho\,X\bigr)\right\|_{L^p}\bigr)\,.
\end{equation}

\subsubsection{Final estimates for the gradient of the velocity} \label{sss:tang-final}

In this paragraph, we complete the proof of propagation of striated regularity, and exhibit a bound for $\nabla u$ in $L^1_{T_0}(L^\infty)$,
for some time $T_0>0$ depending only on suitable norms of the data. 

First, we want to control the $L^p$ norm of $\nabla X$ which, in light of inequality \eqref{est:dX_2},
requires our bounding  the quantity $\|\pa_X\nabla u\|_{L^p}$.
Here lies the main difficulty, compared to the standard result of propagation of striated regularity for incompressible
flows.
On the one hand, the  part of this term corresponding to  $\mbb{P} u$ may be bounded quite easily 
since $\mbb{P} u\,=\,\mbb{P}w$ (note that  $u-w$ is a gradient) and estimates on the second derivatives are thus
available through  the maximal regularity results that have been proved before. 
On the other hand,  $\Q u$  has a part involving $w$ (which is fine, exactly as before) and another one depending 
on  $P(\rho).$ Here   Proposition \ref{p:tang-d_X} will come into play.
More precisely,  we use relation \eqref{eq:du} to write
\begin{equation} \label{eq:d_XDu}
\pa_X\nabla u\,=\,\pa_X\nabla w\,+\,\pa_X\nabla^2(\Id-\Delta)^{-1}P(\rho)\,.
\end{equation}

Bounding the first term is easy:   thanks to   Paragraph \ref{sss:ub-w} and \eqref{est:X_L^p},  we have
\begin{equation} \label{est:d_XDw}
\left\|\pa_X\nabla w\right\|_{L^p}\,\leq\,\|X\|_{L^\infty}\,\|\nabla^2w\|_{L^{p}}\,\leq\,e^{U}\;\|X_0\|_{L^\infty}\,\|\nabla^2w\|_{L^{p}}\,.
\end{equation}
 Recall that, thanks to Theorem \ref{thm:existence} and especially estimate \eqref{est:N}, the quantity $\|\nabla^2w\|_{L^{p}}$ is in $L^{r}_T,$ and thus in $L^1_T.$

For  the last term in \eqref{eq:d_XDu},  Proposition \ref{p:tang-d_X}, 
guarantees that
\begin{align*}
&\left\|\pa_X\nabla^2(\Id-\Delta)^{-1}P(\rho)\right\|_{L^p}\,\leq\,C\,\left(\tn{\nabla\mc X}_{L^p}
\biggl(1+\frac{\tn{\mc X}_{L^\infty}^{4d-5}\,\tn{\nabla\mc X}_{L^p}}{\bigl(I(\mc X)\bigr)^{4d-4}}\biggr)\,\|P(\rho)\|_{L^\infty}\right. \\
&\qquad\qquad\left.+\,\biggl(1+\frac{\tn{\mc X}_{L^\infty}^{4d-5}\,\tn{\nabla\mc X}_{L^p}}{\bigl(I(\mc X)\bigr)^{4d-4}}\biggr)\,\tn{\pa_{\mc X}P(\rho)}_{L^p}\,+\,
\frac{\tn{\mc X}_{L^\infty}^{4d-4}}{\bigl(I(\mc X)\bigr)^{4d-4}}\,\tn{\nabla\mc X}_{L^p}\,\|P(\rho)\|_{L^\infty}\right)\cdotp
\end{align*}
In view of estimates \eqref{est:N} and \eqref{est:d_XP}, this implies that (taking $C$ larger if need be)
$$\displaylines{
\left\|\pa_X\nabla^2(\Id-\Delta)^{-1}P(\rho)\right\|_{L^p}\,\leq\,C\biggl(\Bigl(\tn{\nabla\mc X}_{L^p}+\,\tn{\div(\rho\,\mc X)}_{L^p}\Bigr)
\biggl(1+\frac{\tn{\mc X}_{L^\infty}^{4d-5}\,\tn{\nabla\mc X}_{L^p}}{\bigl(I(\mc X)\bigr)^{4d-4}}\biggr)\hfill\cr\hfill
+\frac{\tn{\mc X}_{L^\infty}^{4d-4}}{\bigl(I(\mc X)\bigr)^{4d-4}}\,\tn{\nabla\mc X}_{L^p}\,\biggr)\cdotp}
$$

At this point, we use the bounds \eqref{est:X_L^p}, \eqref{est:I} and \eqref{est:rho-tang}. Including a dependence  on
$\tn{\mc X_0}_{L^\infty}$, $I(\mc X_0)$ and $\tn{\div\bigl(\rho_0\,\mc X_0\bigr)}_{L^p}$ in $C,$ we deduce
for some constant $c_d$ depending only on $d,$
\begin{align*}
\left\|\pa_X\nabla^2(\Id-\Delta)^{-1}P(\rho)\right\|_{L^p}\,&\leq\,C\,\biggl(\tn{\nabla\mc X}_{L^p}
\left(1+e^{c_dU}\,\tn{\nabla\mc X}_{L^p}\right) \\
&\qquad\qquad+\,\left(1+e^{c_dU}\,\tn{\nabla\mc X}_{L^p}\right)\,e^{U}\,+\,e^{(c_d+1)U}\,\tn{\nabla\mc X}_{L^p}\biggr)\cdotp
\end{align*}
 Changing $c_d$ to $c_d+1$, the previous relation finally leads us to the bound
\begin{equation} \label{est:d_X-CZ}
\left\|\pa_X\nabla^2(\Id-\Delta)^{-1}P(\rho)\right\|_{L^p}\,\leq\,C\,e^{c_dU}\,\left(1\,+\,\tn{\nabla\mc X}^2_{L^p}\right)\,.
\end{equation}

Putting estimates \eqref{est:d_XDw} and \eqref{est:d_X-CZ} together, we finally gather
\begin{equation} \label{est:d_XDu}
\left\|\pa_X\nabla u\right\|_{L^p}\,\leq\,C\,e^{c_dU}\,\left(1\,+\,\tn{\nabla\mc X}^2_{L^p}\,+\,\|\nabla^2w\|_{L^{p}}\right)\,.
\end{equation}

At this point, one define the time $T_1$ to satisfy  
\begin{equation} \label{def:T_0}
T_1\,:=\,\sup\bigl\{t\in\,]0,T]\quad\Bigl|\quad U(t)\,\leq\,\log2
\bigr\}\,,
\end{equation}
where $T>0$ is the time given by  Proposition \ref{p:exist_w-r}.
\medbreak
Then, changing $C$ if need be, estimate \eqref{est:d_XDu} implies that, on $[0,T_1]$, one has
$$
\left\|\pa_X\nabla u\right\|_{L^p}\,\leq\,C\,\left(1\,+\,\tn{\nabla\mc X}^2_{L^p}\,+\,\|\nabla^2w\|_{L^{p}}\right)\cdotp
$$
Inserting now this bound in \eqref{est:dX_2} and using also \eqref{est:N}, we find that for all $\lambda\in\{1,\cdots,m\},$
$$\begin{aligned}
\left\|\nabla X_\lambda(t)\right\|_{L^p}&\,\leq\,2\,\biggl(\|\nabla X_{0,\lambda}\|_{L^p}\,+\,\int^t_0\|\pa_{X_\lambda}\nabla u\|_{L^p}\,d\tau\biggr) \\
\,&\leq\,2\biggl(\|\nabla X_{0,\lambda}\|_{L^p}\,
+\,C\int^t_0\bigl(1\,+\,\tn{\nabla\mc X}^2_{L^p}\,+\,\|\nabla^2w\|_{L^p}\bigr)\,d\tau\biggr)
\cdotp 
\end{aligned}$$
Taking the supremum on $\lambda,$ we find that for some $C_0$ depending only on the norms
of the data, and for some `absolute' constant $C,$ we have for all $t\in[0,T_1],$
$$
\left\|\nabla\mc X(t)\right\|_{L^p}\,\leq\,C\,\biggl(C_0\,+\,C\int^t_0\tn{\nabla\mc X(\tau)}^2_{L^p}\,d\tau\biggr)\cdotp
$$
Then using a Gronwall type argument, we conclude that 
\begin{equation}\label{eq:DXt}
\left\|\nabla\mc X(t)\right\|_{L^p}\,\leq\,\frac{C_0}{1-CC_0t}
\quad\hbox{for all }\  t\in[0,T_1]\ \hbox{  satisfying }\ CC_0t<1.\end{equation}

This having been established, let us turn our attention to finding a bound for the quantity $U(t),$ as it is needed  to close the estimates.
Resorting to relation \eqref{eq:du} again, we see that we have to control the $L^1_t(L^\infty)$ norm of each term appearing on its right-hand side.
For the term in $w$, this is an easy task: thanks to Proposition \ref{p:exist_w-r} and Sobolev embeddings, a decomposition in low and high frequencies implies
\begin{equation} \label{est:Dw_inf}
\left\|\nabla w\right\|_{L^{1}_t(L^\infty)}\,\leq\,C\,\left(t^{1-1/r_1}\,\left\|\nabla w\right\|_{L^{r_1}_T(L^{p})}\,+\,t^{1-1/r}\,\left\|\nabla^2w\right\|_{L^{r}_T(L^{p})}\right)
\end{equation}
for all $t\in[0,T]$, where we also used relation \eqref{index:choice}. 

Bounding the latter term in \eqref{eq:du} is based on Proposition \ref{p:tang} that gives
$$
\left\|\nabla^2(\Id-\Delta)^{-1}P(\rho)\right\|_{L^\infty}\,\leq\,C\biggr(
\biggl(1+\frac{\tn{\mc X}_{L^\infty}^{4d-5}\,\tn{\nabla\mc X}_{L^p}}{\bigl(I(\mc X)\bigr)^{4d-4}}\biggr)\,\|P(\rho)\|_{L^\infty}+\,\frac{\tn{\mc X}_{L^\infty}^{4d-5}}{\bigl(I(\mc X)\bigr)^{4d-4}}\,\tn{\pa_{\mc X}P(\rho)}_{L^p}\biggl)\cdotp
$$
In view of \eqref{est:N}, \eqref{est:X_L^p}, \eqref{est:I}, \eqref{est:rho-tang} and \eqref{est:d_XP}, omitting once again the explicit dependence of the multiplicative constants on the norms
of the initial data, the previous inequality allows us to get
\begin{align}
\left\|\nabla^2(\Id-\Delta)^{-1}P(\rho)\right\|_{L^\infty}\,&\leq\,C\,\left(1+e^{c_dU}\,\tn{\nabla\mc X}_{L^p}\,+\,e^{c_dU}\,\left(\tn{\nabla\mc X}_{L^p}\,+\,
e^{U}\right)\right) \label{est:DCZ_inf} \\
&\leq\,C\,e^{c_dU}\,\bigl(1\,+\,\tn{\nabla\mc X}_{L^p}\bigr)\,. \nonumber
\end{align}
Recalling definition \eqref{def:T_0} of $T_1$ and Inequality \eqref{eq:DXt}
and taking $0<T_0\leq T_1$ so that $2CC_0T_0\leq1,$ we gather
\begin{equation} \label{est:DCZ_inf-2}
 \left\|\nabla^2(\Id-\Delta)^{-1}P(\rho)\right\|_{L^\infty_{T_0}(L^\infty)}\,\leq\,C\,\bigl(1\,+\,C_0\bigr)\,,
\end{equation}
which implies, together with \eqref{est:Dw_inf}, the following control, for all fixed $t\in[0,T_0]$:
\begin{equation} \label{est:Du-L^inf}
\left\|\nabla u\right\|_{L^1_t(L^\infty)}\,\leq\,C\left(t^{1-1/r_1}\left\|\nabla w\right\|_{L^{r_1}_{T}(L^{p_1})}+t^{1-1/r}\left\|\nabla^2w\right\|_{L^{r}_T(L^{p})}+
t\bigl(1+ C_0\bigr)\right)\,.
\end{equation}
Up to taking a smaller $T_0$, we then see that the requirement $\left\|\nabla u\right\|_{L^1_{T_0}(L^\infty)}\,\leq\,\log2$ can be fulfilled.
Then, a classical bootstrap argument, which we do not detail here, finally allows us to deduce the boundedness of $\nabla u$ in $L^1_{T_0}(L^\infty)$.
\smallbreak
In order to prove rigorously the existence part of Theorem \ref{thm:reg}, 
one may proceed as in Subsection \ref{ss:existence}. 
There, we constructed a sequence $(\vrho^n,u^n)$ of smooth solutions that is uniformly bounded
in the space $E_T.$  Therefore, it is only a matter of checking that 
one can get uniform bounds, too, for the striated regularity. 
To do this, we smooth out the reference family of vector fields $\cX_0$ into $\cX_0^n$ 
(paying attention to keep the nondegeneracy condition), then 
define the family $\cX^n:=\bigl(X^n_\lambda)_{1\leq\lambda\leq m}$ transported by the
flow of $u^n$ according to \eqref{eq:X}, taking $u^n$ instead of $u$ and starting
from the initial vector field $X_{0,\lambda}^n.$ 
Then, repeating the computations that have been carried out just above, 
we get uniform bounds for all the quantities involving
the striated regularity, and thus also for $\nabla u^n$ in $L^1_{T_0}(L^\infty).$
That $(\vrho^n,u^n)$ tends to some solution $(\vrho,u)$ of \eqref{NSbis} belonging 
to $E_{T_0}$ has already been justified before. 
Furthermore, combining our new bounds with compactness arguments
allows to pass to the limit in \eqref{eq:X} as well, 
and to get the crucial information that $\nabla u$  is in $L^1_{T_0}(L^\infty).$


\subsection{The proof of uniqueness} \label{ss:unique}

Property \eqref{est:Du-L^inf} having been established, one can now tackle the proof of uniqueness of solutions.
The basic idea  is  to perform a Lagrangian change of coordinates in system \eqref{NS}, in order to by-pass the hyperbolic nature of the mass equation, which otherwise would cause the loss of
one derivative in the stability estimates. 
In fact, we will perform stability estimates for  the Lagrangian formulation of equations \eqref{NS}.

\subsubsection{Lagrangian formulation} \label{sss:Lagr}
The goal of this paragraph is to recast system \eqref{NS} in Lagrangian variables. 
Recall that in light of  the estimates of Subsection \ref{ss:tangential}, we know that
for all $k_0>0,$ there exists a time $T_0>0$ such that 
\begin{equation} \label{est:Du-small}
\int_0^{T_0}\left\|\nabla u(t)\right\|_{L^\infty}\,dt\,\leq\,k_0\,.
\end{equation}
The value of  $k_0$  will be determined in the course of the  computations below.
\medbreak
First of all, we define the flow $\psi_u$ associated to the velocity field $u$ to be the solution of
\begin{equation} \label{def:flow}
\psi_u(t,y)\,:=\,y\,+\,\int_0^tu\bigl(\tau,\psi_u(\tau,y)\bigr)\,d\tau\,.
\end{equation}
Thanks to that,   any function $f=f(t,x)$ may be rewritten in Lagrangian coordinates $(t,y)$ according to the relation
\begin{equation}\label{eq:lag}\oline{f}(t,y)\,:=\,f\bigl(t,\psi_u(t,y)\bigr)\cdotp\end{equation}
A key observation is that, once passing in Lagrangian coordinates, one can forget about the reference  \emph{Eulerian} velocity $u$ by rewriting Definition
\eqref{def:flow} in terms of the \emph{Lagrangian} velocity $\oline u$, defining directly $\psi_u$ by
$$
\psi_u(t,y)\,=\,y\,+\,\int^t_0\oline u(\tau,y)\,d\tau\,.
$$
In what follows, we  set $J_u\,:=\,\det\bigl(D\psi_u\bigr)$ and $A_u\,:=\,\bigl(D\psi_u\bigr)^{-1}$. Observe that, by the standard chain rule, we obviously have
\begin{equation} \label{eq:grad_lagr}
\oline{D_xf}\,=\,D_y\oline f\cdot A_u\,.
\end{equation}
 Lemma A.2 of \cite{D-Fourier} provides us with the following  alternative  expressions\footnote{From now 
 on  we agree that  $\adj(M)$ designates  the adjugate matrix of $M.$ Of course, if $M$ is invertible, then $\adj(M)\,=\,(\det M)\,M^{-1}$.}.
 \begin{lem} \label{l:change}
For any $\mc C^1$ function $K$ and any $\mc C^1$ vector-field $H$ defined over $\R^d$, one has
\begin{align*}
\oline{\nabla_xK}\,&=\,J_u^{-1}\,\div_y\bigl(\adj D\psi_u\,\oline K\bigr) \\
\oline{\div_xH}\,&=\,J_u^{-1}\,\div_y\bigl(\adj D\psi_u\,\oline H\bigr)\,.
\end{align*}
\end{lem}
Moreover, since our diffeomorphism $\psi_u$ is the flow of the time-dependent vector-field $u$, we also get, for any function $f$,
\begin{equation}\label{eq:masslag}
\oline{\pa_tf\,+\,\div(f\,u)}\,=\,J_u^{-1}\,\pa_t\bigl(J_u\,\oline f\bigr)\,.
\end{equation}

The next statement is in the spirit of Lemma A.3 of \cite{D-Fourier}; also its proof follows the same steps, up to 
a straightforward adaptation to our functional framework.
\begin{lem} \label{l:est_A-J}
Let $u$ be a velocity field with $\nabla u\,\in\,L^1\bigl([0,T_0];L^\infty(\R^d)\bigr)$, and let $\psi_u$ be its flow, defined by \eqref{def:flow}.
Suppose that condition \eqref{est:Du-small} is fulfilled with $k_0<1.$

Then there exists a constant $C>0$, just depending on $k_0$, such that the following estimates hold true, for all time $t\in[0,T_0]$:
\begin{align*}
\left\|\Id\,-\,\adj D\psi_u(t)\right\|_{L^\infty}\,&\leq\,C\,\left\|Du\right\|_{L^1_{T_0}(L^\infty)} \\
\left\|\Id\,-\,A_u(t)\right\|_{L^\infty}\,&\leq\,C\,\left\|Du\right\|_{L^1_{T_0}(L^\infty)} \\
\left\|J_u^{\pm1}(t)\,-\,1\right\|_{L^\infty}\,&\leq\,C\,\left\|Du\right\|_{L^1_{T_0}(L^\infty)}\,.
\end{align*}
\end{lem}

We also state the following lemma, the  proof of which is straightforward.
\begin{lem} \label{l:L^p-Euler-vs-Lagr}
For any function $f=f(x)$, define $\oline f$ according to \eqref{eq:lag}. Then for any $p\in[1,\infty[\,$ one has
$$
\left\|\oline{f}\right\|_{L^p}\,\leq\,\left\|J_u\right\|^{1/p}_{L^\infty}\,\left\|f\right\|_{L^p}\qquad\mbox{ and }\qquad
\left\|f\right\|_{L^p}\,\leq\,\left\|J_u^{-1}\right\|^{1/p}_{L^\infty}\,\left\|\oline f\right\|_{L^p}\,.
$$
\end{lem}

After these preliminaries, we can recast our system in Lagrangian coordinates. First of all, from the mass equation in \eqref{NS} and 
\eqref{eq:masslag}, we discover that
\begin{equation} \label{eq:Lagr-rho}
\partial_t(J_u\,\oline{\rho})\,=\,0,\quad\hbox{whence}\quad J_u\,\overline\rho\,=\,\rho_0\,.
\end{equation}
Second, we notice that, in Lagrangian coordinates, operator $\mc L$ reads
\begin{equation} \label{eq:L-tilde}
 \wtilde{\mc L}\,f\,:=\,-\,J_u^{-1}\bigl(\mu\,\div\left(\adj(D\psi_u)\,^t\!A_u\,\nabla f\right)\,-\,\lambda\,\div\bigl(\adj(D\psi_u)\,(\,^t\!A_u:\nabla f)\bigr)\bigr)
\end{equation}
where we have used the notation $M:N\,:=\,tr(MN)\,=\,\sum_{ij} M_{ij}\,N_{ji}$.
\smallbreak
Hence, thanks to  \eqref{eq:masslag} and \eqref{eq:Lagr-rho}, the momentum equation in \eqref{NS} recasts in \begin{equation} \label{eq:Lagr-u}
\rho_0\,\p_t \oline{u}\,+\,\wtilde{\mc L}\oline{u}\,=\,-\,\div\bigl(\adj D\psi_u\, P(J^{-1}\,\rho_0)\bigr)\cdotp
\end{equation}

\subsubsection{Stability estimates in Lagrangian coordinates} \label{sss:stab}

In this section, we tackle the proof of uniqueness, by showing stability estimates for the Lagrangian formulation of our system.

More precisely, we consider initial data $\bigl(\rho_0^j,u_0^j\bigr)$, for $j=1,2$, verifying the hypotheses of Theorem \ref{thm:reg}.
For the sake of simplicity and clarity, we focus on the case $\rho_0^1\,=\,\rho_0^2\,=\,\rho_0$, and suppose that $\rho_0$ satisfies the striated regularity assumption
with respect to some fixed non-degenerate family of vector-fields $\mc X_0$.
The initial velocities do not need to be equal.

Let $\bigl(\rho^1,u^1\bigr)$ and $\bigl(\rho^2,u^2\bigr)$ be two solutions to system \eqref{NS} 
on the time interval $[0,T],$ fulfilling the properties given by Theorem \ref{thm:reg}
and corresponding to the data  $\bigl(\rho_0,u_0^1\bigr)$ and $\bigl(\rho_0,u_0^2\bigr),$ respectively. 
Denoting  $\vrho^{j}\,=\,\rho^j-1$   for $j=1,2,$ and defining  $w^{j}$  according to \eqref{d:w},   the pairs
 $\bigl(\vrho^j,w^j\bigr)$  solve equations \eqref{eq:vrho}-\eqref{eq:w} and also enjoy the regularity properties stated in Theorem \ref{thm:existence}.
Moreover, as shown in Subsection \ref{ss:tangential}, for all $j$, tangential regularity is propagated with respect to the  non-degenerate family $\mc X^{j}$, which correspond to the family
$\mc X_0$ transported by the flow of $u^{j}.$
Hence, for all $k_0>0,$ there exists $T_0>0$ such that 
both $\nabla u^1$ and $\nabla u^2$  fulfill \eqref{est:Du-small}, which allows to   pass
in Lagrangian coordinates, as shown in Paragraph \ref{sss:Lagr}. Denoting, for $j=1,2$, the flow 
of $u^j$ by $\psi_j,$   setting $J_j:=J_{u^j}$ and
$A_j:=A_{u^j}$,  and taking advantage of  the previous computations, we discover that $\bigl(\oline\rho^j,\oline u^j\bigr)_{j=1,2}$ satisfy the relations $J_j\,\overline\rho^j\,=\,\rho_0$ and
$$
\rho_0\,\p_t \oline{u}^j\,+\,\wtilde{\mc L}_j\oline{u}^j\,=\,-\,\div\bigl(\adj D\psi_j\, P(J_j^{-1}\,\rho_0)\bigr)\,,
$$
where $\wtilde{\mc L}_j$ is the operator corresponding to $u^j$ that has been  defined by formula \eqref{eq:L-tilde}.

Let  $\d\oline{u}\,:=\,\oline{u}^1-\oline{u}^2$ and use similar notations for the other quantities. Taking the difference of the equations respectively for
$\oline{u}^1$ and $\oline{u}^2$, we find that $\delta\oline{u}$ satisfies
\begin{align} \label{eq:delta-u}
\rho_0\,\p_t \d\oline{u}\,+\,\mc L\d\oline{u}\,&=\,\left(\mc L-\wtilde{\mc L}_1\right)\,\d\oline{u}\,+\,\d\mc{L}\,\oline{u}^2\,-\,\div\bigl(\d\adj\,P(J_1^{-1}\,\rho_0)\bigr) \\
&\qquad\qquad\qquad-\,\div\Bigl(\adj D\psi_2\,\bigl(P(J_1^{-1}\,\rho_0)\,-\,P(J_2^{-1}\,\rho_0)\bigr)\Bigr)\,,  \nonumber
\end{align}
where we have set $\d\adj\,:=\,\adj D\psi_1-\adj D\psi_2$.
A slight adaptation of Lemma A.4 of \cite{D-Fourier} allows us to get the following bounds.
\begin{lem} \label{l:est_delta_A-J} If \eqref{est:Du-small} is fulfilled by $u^1$ and $u^2$
for some $k_0\in]0,1],$ then there exists a constant $C>0$ just depending on $k_0$, such that the following estimates hold true, for all time $t\in[0,T_0]$ and all $p\in[1,\infty]$:
\begin{align*}
\left\|\adj D\psi_1(t)\,-\,\adj D\psi_2(t)\right\|_{L^p}\,&\leq\,C\,\int_0^t\left\|\nabla\d u(\tau)\right\|_{L^p}\,d\tau \\
\left\|A_1(t)\,-\,A_2(t)\right\|_{L^p}\,&\leq\,C\,\int_0^t\left\|\nabla\d u(\tau)\right\|_{L^p}\,d\tau \\
\left\|J_1^{\pm1}(t)\,-\,J_2^{\pm1}(t)\right\|_{L^p}\,&\leq\,C\,\int_0^t\left\|\nabla\d u(\tau)\right\|_{L^p}\,d\tau \,.
\end{align*}
\end{lem}

We now perform energy estimates for equation \eqref{eq:delta-u}: take the $L^2$
scalar product of  both sides with $\delta\oline{u}$ and integrate by parts. 
In view of Lemmas \ref{l:est_A-J} and \ref{l:est_delta_A-J}, we deduce the following controls for the terms coming from the right-hand side:
\begin{align*}
\left|\int \left(\mc L-\wtilde{\mc L}_1\right)\,\d\oline{u}\,\cdot\,\d\oline{u}\, dx\right|\, &\leq\,C\,k_0\,\left\|\nabla\d\oline{u}\right\|^2_{L^2} \\
\left|\int \d\mc{L}\,\oline{u}^2\,\cdot\,\d\oline{u}\, dx\right|\, &\leq \,C\biggl(\int^t_0\|\nabla\d\oline u\|_{L^2}\,d\tau\biggr)\left\|\nabla\oline{u}^2\right\|_{L^\infty}\,\left\|\nabla\d\oline{u}\right\|_{L^2} \\
\left|\int \div\bigl(\d\adj\,P(J_1^{-1}\,\rho_0)\bigr)\,\cdot\,\d\oline u\, dx\right|\, &\leq \,C\biggl(\int^t_0\|\nabla\d\oline u\|_{L^2}\,d\tau\biggr)\|\rho_0\|_{L^\infty}\,\left\|\nabla\d\oline{u}\right\|_{L^2} \\
\left|\int \div\Bigl(\adj D\psi_2\,\bigl(P(J_1^{-1}\,\rho_0)\,-\,P(J_2^{-1}\,\rho_0)\bigr)\Bigr)\,\cdot\,\d\oline{u}\,dx\right|\, &\leq\,
C\biggl(\int^t_0\|\nabla\d\oline u\|_{L^2}\,d\tau\biggr)\|\rho_0\|_{L^\infty}\,\left\|\nabla\d\oline{u}\right\|_{L^2}\,.
\end{align*}
Now, if $k_0$ in  \eqref{est:Du-small} has been taken small enough, then a repeated use of Young and Cauchy-Schwarz inequalities leads
 us to the estimate
\begin{align}
\frac{d}{dt}\int\rho_0\,|\d\oline{u}|^2\,dx\,+\,\int|\nabla\d\oline{u}|^2\,dx\,\leq\,C\,t\,\left(1+\left\|\nabla\oline{u}^2\right\|^2_{L^\infty}\right)\,
\int^t_0\|\nabla\d\oline u\|^2_{L^2}\,d\tau\, \label{est:uniq_prelim}
\end{align}
for a new constant $C>0$ that depends only  on $k_0,$  $\|\rho_0\|_{L^\infty},$ $P,$ $\lambda$ and $\mu.$
\smallbreak
In order to conclude to uniqueness on $[0,T_0]$ by applying Gronwall lemma,  we need  that 
\begin{equation}\int_0^{T_0} t\,\left(1+\left\|\nabla\oline{u}^2\right\|^2_{L^\infty}\right)dt<\infty.\end{equation}
In view of \eqref{eq:grad_lagr} and Lemma \ref{l:L^p-Euler-vs-Lagr}, it suffices  to show that 
$t^{1/2}\nabla u^2$ is in $L^2_{T_0}(L^\infty).$ Now,  recall that
$$\nabla u^2\,=\,\nabla w^2\,-\,\nabla^2(\Id-\Delta)^{-1}P(\rho^2)\,,$$
and that, thanks to the estimates of Subsection \ref{ss:tangential}, one has that
$\nabla^2(\Id-\Delta)^{-1}P(\rho^2)$ belongs to $L^\infty\bigl([0,T_0]\times\R^d\bigr)$.
Therefore, our next (and final) goal is to show that there exists $T_0>0$ such that
\begin{equation}\int_0^{T_0} t\left\|\nabla w^2\right\|^2_{L^\infty}\,dt<\infty.\end{equation}
This will be achieved thanks to the following proposition, the proof of which is postponed to the next paragraph.
\begin{prop} \label{p:max-reg_new}
Under the hypotheses of Theorem $\ref{thm:reg},$ let us fix some
$$
R_2\,>\,\max\left\{r_0/2\,,\,r_1/2\,,\,2\right\}\qquad\mbox{ such that }\qquad \frac{1}{2\,R_2}\,<\,\frac{3}{2}\,-\,\frac{1}{r_2}\,+\,\frac{d}{2\,p}\with r_2:=r,
$$
and set $R_0\,=\,R_1\,=\,2\,R_2$. For $j\in\{0,1,2\}$, define moreover the indices
\begin{equation}\label{def:gamma}
\alpha_j\,:=\,\frac{1}{r_j}\,-\,\frac{1}{R_j}\qquad\mbox{ and }\qquad \g_j\,:=\,\frac{1}{r_j}\,-\,\frac{1}{2\,R_j}\cdotp\end{equation}
Then, there exists a positive time $T_*$ such that one has the properties
$$
t^{\alpha_2}\,\nabla^2w\,\in\,L^{R_2}_{T_*}(L^{p})\;,\qquad t^{\g_1}\,\nabla w\,\in\,L^{R_1}_{T_*}(L^{p})\;,\qquad t^{\g_0}\,w\,\in\,L^{R_0}_{T_*}(L^{\infty})\,.
$$
\end{prop}

From the previous proposition, we immediately deduce the following corollary.
\begin{col} \label{c:mr_new}
Under the assumptions of Theorem \ref{thm:reg}, one has
$$
\int^{T_*}_0t\,\|\nabla w(t)\|^2_{L^\infty}\,dt\,<\,\infty\,.
$$
\end{col}

\begin{proof}
By Sobolev embedding, the stated inequality is a consequence of the following computation:
\begin{align*}
\int^{T_*}_0t\,\|\nabla^2 w(t)\|^2_{L^{p}}\,dt\,&=\,\int^{T_*}_0t^{1-\eta}\,t^\eta\,\|\nabla^2 w(t)\|^2_{L^{p}}\,dt \\
&\leq\,\left(\int^{T_*}_0t^{\eta\,q}\,\|\nabla^2 w(t)\|^{2\,q}_{L^{p}}\,dt\right)^{\!\!1/q}\,\left(\int^{T_*}_0t^{(1-\eta)\,q'}\,dt\right)^{\!\!1/q'}\,,
\end{align*}
where $q'$ is the conjugate exponent of $q$. Take $q\,=\,R_2/2$, so that $1/q'\,=\,1\,-\,2/R_2$ and  impose the relation
$q\,\eta\,=\,R_2\,\alpha_2$, getting in this way $\eta\,=\,2\,\alpha_2$. With these choices and because $r_2=r\in(1,2),$ 
the condition $(1-\eta)\,q'>-1$ is  verified, which completes the proof of the corollary.
\end{proof}

At this point, one can finish the proof of Theorem \ref{thm:reg}, by establishing the uniqueness of solutions. Let us define the function
$$
E(t)\,:=\,\left\|\sqrt{\rho_0}\;\delta\oline{u}(t)\right\|_{L^2}^2\,+\,\int^t_0\|\nabla\d \oline{u}(\tau)\|_{L^2}^2\,d\tau\,.
$$
Up to choosing a smaller $T_0$, we can suppose that $T_0=T_*$. Then, applying Gronwall inequality to \eqref{est:uniq_prelim}, we get, for all $t\in[0,T_*]$, the bound
$$
E(t)\,\leq\,E(0)\,\exp\left(C\,\int^t_0f(\tau)\,d\tau\right)\,,\qquad\mbox{ where }\qquad f(t)\,:=\,t\,\left(1+\left\|\nabla\oline{u}^2(t)\right\|^2_{L^\infty}\right)\cdotp
$$
Since $E(0)\equiv0$  and, by Corollary \ref{c:mr_new}, $f\in L^1([0,T_*]),$ we get
uniqueness on $[0,T_*].$ Combining with a standard 
continuation argument, we then conclude to uniqueness on the whole interval $[0,T].$

\subsubsection{Maximal regularity with time weights}

For completeness, we still have to prove  Proposition \ref{p:max-reg_new}. As a
first, we need the following  lemma that concerns the maximal regularity issue with time weights for the
heat semi-group, and  is strongly inspired by Lemma 3.2 of \cite{HPZ}.
\begin{lem} \label{l:Dw-w_new}
Let the exponents $\left(R_j\,,\,\alpha_j\,,\,\g_j\right)_{j\in\{0,1,2\}}$ be chosen as in Proposition \ref{p:max-reg_new}. Let the operators $\mc A_1$ and $\mc A_0$ be defined as in
Lemmas \ref{l:max-reg_1} and \ref{l:max-reg_0}. Fix some $T>0$, and assume that $t^{\alpha_2}\,f$ belongs to $L^{R_2}_T(L^{p})$.

Then one has $\;t^{1/r_1}\,\mc{A}_1f\;\in\;L^{\infty}_T(L^{p})\;$ and $\;t^{1/r_0}\,\mc A_0f\;\in\;L^{\infty}_T(L^\infty)$, together with the estimates
\begin{align*}
\left\|t^{1/r_1}\,\mc A_1f\right\|_{L^\infty_T(L^{p})}\,+\,\left\|t^{1/r_0}\,\mc A_0f\right\|_{L^\infty_T(L^{\infty})}\,&\leq\,
C\,\left\|t^{\alpha_2}\,f\right\|_{L^{R_2}_T(L^{p})}\,.
\end{align*}
Moreover,  we have  $\;t^{\alpha_1}\,\mc A_1f(t)\;\in\;L_T^{R_1/(1+\d)}(L^p)$ and
$\;t^{\alpha_0}\,\mc A_0f(t)\;\in\;L_T^{R_0/(1+\d)}(L^\infty)$ for all $\d>0$, with the bounds
\begin{align*}
\left\|t^{\alpha_1}\,\mc A_1f\right\|_{L^{R_1/(1+\d)}_T(L^{p})}\,+\,\left\|t^{\alpha_0}\,\mc A_0f\right\|_{L^{R_0/(1+\d)}_T(L^{\infty})}\,&\leq\,
C\;\left(T^{\d/R_1}\,+\,T^{\d/R_0}\right)\,\left\|t^{\alpha_2}\,f\right\|_{L^{R_2}_T(L^{p})}\,.
\end{align*}
In particular, defining $\gamma_0$ and $\gamma_1$  according to \eqref{def:gamma}, we have $\;t^{\g_1}\,\mc A_1f\;\in\;L_T^{R_1}(L^p)$ and $\;t^{\g_0}\,\mc A_0f\;\in\;L_T^{R_0}(L^\infty)$, and the following estimate is verified:
\begin{align*}
\left\|t^{\g_1}\,\mc A_1f\right\|_{L^{R_1}_T(L^{p})}\,+\,\left\|t^{\g_0}\,\mc A_0f\right\|_{L^{R_0}_T(L^{\infty})}\,&\leq\,
C\;\left(T^{1/(2R_1)}\,+\,T^{1/(2R_0)}\right)\,\left\|t^{\alpha_2}\,f\right\|_{L^{R_2}_T(L^{p})}\,.
\end{align*}
\end{lem}
\begin{proof}
Regarding operator $\mc A_1,$ going along the lines of the proof to Lemma \ref{l:max-reg_1}, one gets 
$$
\left\|\nabla e^{(t-s)\Delta}f(s,\cdot)\right\|_{L^{p}}\,\leq\,C\,(t-s)^{-1/2}\,\|f(s)\|_{L^{p}}\quad\hbox{for all }\ 
0\leq s\leq t\leq T,
$$
which implies, after setting $1/R_2'\,=\,1\,-\,1/R_2$, the inequality
\begin{align*}
\left\|\mc A_1f(t)\right\|_{L^{p}}\,&\leq\,C\,\int^t_0(t-s)^{-1/2}\,s^{-\alpha_2}\,\left\|s^{\alpha_2}\,f(s)\right\|_{L^{p}}\,ds \\
&\leq\,C\,\left(\int^t_0(t-s)^{-R_2'/2}\,s^{-\alpha_2\,R_2'}\,ds\right)^{1/R_2'}\;\left\|s^{\alpha_2}\,f\right\|_{L^{R_2}_T(L^{p})}\,.
\end{align*}
Since $R_2>2$, we have $R_2'/2<1$, while, by our definition of $\alpha_2$ in \eqref{def:gamma}, 
we have $\alpha_2\,R_2'<1.$
Therefore,  performing the change of variable $s\,=\,t\,\tau$ inside the integral yields
$$\left\|\mc A_1f(t)\right\|_{L^{p}}\,\leq\,C\,t^{1/2\,-\,\alpha_2\,-\,1/R_2}\,\left\|s^{\alpha_2}\,f\right\|_{L^{R_2}_T(L^{p})}\,.
$$
On the one hand, since  $1/2-\alpha_2-1/R_2\,=\,1/2-1/r_2\,=\,-1/r_1$, we have 
\begin{equation} \label{est:A_1f-weight_1}
\left\|t^{1/r_1}\,\left\|\mc A_1f(t)\right\|_{L^{p}}\right\|_{L^\infty_T}\,\leq\,C\,\left\|s^{\alpha_2}\,f\right\|_{L^{R_2}_T(L^{p})}\,.
\end{equation}
On the other hand, since $1/2-\alpha_2-1/R_2\,=\,-\,\alpha_1\,-\,1/R_1$, we also get that $t^{\alpha_1}\,\left\|\mc A_1f(t)\right\|_{L^{p}}$ belongs to
$L_T^{R_1/(1+\d)}$ for all $\d>0$, and verifies
\begin{equation} \label{est:A_1f-weight_2}
\Bigl\|t^{\alpha_1}\,\left\|\mc A_1f(t)\right\|_{L^{p}}\Bigr\|_{L^{R_1/(1+\d)}_T}\,\leq\,C_\d\,\left\|s^{\alpha_2}\,f\right\|_{L^{R_2}_T(L^{p})}\,T^{\d/R_1}\,.
\end{equation}
Taking  $\d=1$ and interpolating between estimates \eqref{est:A_1f-weight_1} and \eqref{est:A_1f-weight_2}, we get that
$t^{\g_1}\,\left\|\mc A_1f(t)\right\|_{L^{p}}\,\in\,L^{R_1}_T$, with the estimate
\begin{equation} \label{est:A_1f-weight_f}
\Bigl\|t^{\g_1}\,\left\|\mc A_1f(t)\right\|_{L^{p}}\Bigr\|_{L^{R_1}_T}\,\leq\,C\,\left\|s^{\alpha_2}\,f\right\|_{L^{R_2}_T(L^{p})}\,T^{1/(2R_1)}\,.
\end{equation}

Proving the claimed bound for the term $\mc A_0$  follows from the same lines. First of all, setting $p'$ to be the conjugate exponent of $p$, we can write
\begin{align*}
\left\|e^{(t-s)\Delta}f(s,\cdot)\right\|_{L^{\infty}}\,&\leq\,C\,(t-s)^{-d/2}\,\biggl\|K_0\biggl(\f{\cdot}{\sqrt{4\pi(t-s)}}\biggr)\biggr\|_{L^{p'}}\,\|f(s)\|_{L^{p}} \\
&\leq\,C\,(t-s)^{-d/(2p)}\,s^{-\alpha_2}\,\|s^{\alpha_2}\,f(s)\|_{L^{p}}\,.
\end{align*}
Integrating this expression in time and applying H\"older inequality once give us, similarly as above,
$$
\left\|\mc A_0f(t)\right\|_{L^\infty}\,\leq\,C\,\left(\int^t_0(t-s)^{-d/(2p)\,R_2'}\,s^{-\alpha_2\,R_2'}\,ds\right)^{1/R_2'}\;\left\|s^{\alpha_2}\,f\right\|_{L^{R_2}_T(L^{p})}\,.
$$
Once again, thanks to our choice of $R_2$ we have that $d/(2p)\,R_2'\,<1$ (recall that $p>d$); hence, repeating the change of variable $s\,=\,t\,\tau$ we find
$$
\left\|\mc A_0f(t)\right\|_{L^\infty}\,\leq\,C\,t^{1\,-\,d/(2p)\,-\,\alpha_2\,-\,1/R_2}\,\left\|t^{\alpha_2}\,f\right\|_{L^{R_2}_T(L^{p})}\,.
$$
Now, first we remark that $1-d/(2p)-\alpha_2-1/R_2\,=\,-1/r_0$, and hence $t^{1/r_0}\,\left\|\mc A_0f(t)\right\|_{L^{\infty}}\,\in\,L^\infty_T$, with
\begin{equation} \label{est:A_0f-weight_1}
\left\|t^{1/r_0}\,\left\|\mc A_0f(t)\right\|_{L^{\infty}}\right\|_{L^\infty_T}\,\leq\,C\,\left\|s^{\alpha_2}\,f\right\|_{L^{R_2}_T(L^{p})}\,.
\end{equation}
Then, we also notice that $1-d/(2p)-\alpha_2-1/R_2\,=\,-\alpha_0-1/R_0$, so that $t^{\alpha_0}\,\left\|\mc A_0f(t)\right\|_{L^{\infty}}$ belongs to
$L_T^{R_0/(1+\d)}$ for all $\d>0$, and verifies the estimate
\begin{equation} \label{est:A_0f-weight_2}
\Bigl\|t^{\alpha_0}\,\left\|\mc A_0f(t)\right\|_{L^{\infty}}\Bigr\|_{L^{R_0/(1+\d)}_T}\,\leq\,C_\d\,\left\|s^{\alpha_2}\,f\right\|_{L^{R_2}_T(L^{p})}\,T^{\d/R_0}\,.
\end{equation}
As above, taking $\d=1$ and interpolating between estimates \eqref{est:A_0f-weight_1} and \eqref{est:A_0f-weight_2}, we finally deduce the property
$t^{\g_0}\,\left\|\mc A_0f(t)\right\|_{L^{\infty}}\,\in\,L^{R_0}_T$, together with the estimate
\begin{equation} \label{est:A_0f-weight_f}
\Bigl\|t^{\g_0}\,\left\|\mc A_0f(t)\right\|_{L^{\infty}}\Bigr\|_{L^{R_0}_T}\,\leq\,C\,\left\|s^{\alpha_2}\,f\right\|_{L^{R_2}_T(L^{p})}\,T^{1/(2R_0)}\,.
\end{equation}

The lemma is now proved.
\end{proof}

Finally, we need   the following lemma that has been established in \cite{HPZ}:
\begin{lem} \label{l:D^2w_new}
Let $1\,<\,R,p\,<\,\infty$, and let $\alpha\geq0$ be such that $\alpha\,+\,1/R\,<1$. Suppose that $t^\alpha\,f$ belongs to $L^R_T(L^p)$, for some $T\in\,]0,\infty]$.

Then also $t^\alpha\,\mc A_2f$ belongs to $L^R_T(L^p)$, and one has the estimate
$$
\left\|t^{\alpha}\,\mc A_2f\right\|_{L^R_T(L^p)}\,\leq\, C\left\|t^\alpha\,f\right\|_{L^R_T(L^p)}\,.
$$
\end{lem}

Now, we are in the position of proving Proposition \ref{p:max-reg_new}.
\begin{proof}[Proof of Proposition \ref{p:max-reg_new}]
Recall that $w$ satisfies \eqref{eq:w}, and thus
\begin{equation}\label{eq:Duh} w(t)= e^{-t\cL}w_0-\int_0^t e^{(s-t)\cL}(\rho F)(s)\,ds\with F\ \mbox{ given by \eqref{d:F}.} \end{equation}
Let us first study the term containing the initial data.
By hypothesis,  $\nabla^2 w_0\,\in\,\dot B^{-2/r}_{p,r}\,\hookrightarrow\,\dot B^{-2/r}_{p,R_2}$, since $R_2>2>r$ by our definitions. Thanks to Proposition \ref{p:max-reg},
this  implies that  $t^{\alpha_2}\,\nabla^2 e^{-t \cL}w_0$ belongs to $L^{R_2}\bigl(\R_+;L^{p}(\R^d)\bigr)$.

In the same way, we have that $\nabla w_0\,\in\,\dot B^{-2/r_1}_{p,r_1}\,\hookrightarrow\,\dot B^{-2/r_1}_{p,R_1}$ and
$w_0\,\in\,\dot B^{-2/r_0}_{\infty,r_0}\,\hookrightarrow\,\dot B^{-2/r_0}_{\infty,R_0}$, because we have taken $R_1=R_0=2R_2>\max\{r_0,r_1\}$. From these properties we deduce that
$t^{\alpha_1}\,\nabla e^{ -t\cL}w_0\,\in\,L^{R_1}\bigl(\R_+;L^{p}(\R^d)\bigr)$ and that $t^{\alpha_0}\,e^{-t\cL}w_0\,\in\,L^{R_0}\bigl(\R_+;L^{\infty}(\R^d)\bigr)$.
Since now both $\g_1$ and $\g_0$ are greater than $\alpha_1$ and $\alpha_0$ respectively, we get that, for all $T>0$ fixed,
$t^{\g_1}\,\nabla e^{-\cL t}w_0\,\in\,L^{R_1}_T(L^{p})$ and $t^{\g_0}\,e^{-t\cL}w_0\,\in\,L^{R_0}_T(L^{\infty})$.

As for the forcing term of \eqref{eq:Duh}, we apply Lemma \ref{l:D^2w_new} with $R = R_2$ and $\alpha=\alpha_2$ (note that $\alpha_2\,+\,1/R_2\,=\,1/r<1$).
We also apply Lemma \ref{l:Dw-w_new}. Therefore, if we set 
$$
\wtilde{\mc N}(T)\,:=\,\|\vrho\|_{L^\infty_T(L^{p}\cap L^\infty)}\,+\,\left\|t^{\g_0}\,w\right\|_{L^{R_0}_T(L^{\infty})}\,+\,\left\|t^{\g_1}\,\nabla w\right\|_{L^{R_1}_T(L^{p})}\,+\,
\left\|t^{\alpha_2}\,\nabla^2w\right\|_{L^{R_2}_T(L^{p})}\,,
$$
arguing exactly as in Paragraph \ref{sss:ub-w},  we get
for some constant $C_T$ bounded by a positive power of $T,$
\begin{equation} \label{est:new-mr_prelim}
\wtilde{\mc N}(T)\,\leq\,C_T\left(\|\vrho_0\|_{L^{p}\cap L^\infty}\,+\,\left\|w_0\right\|_{\dot B^{2-2/r}_{p,r}}\,+\,
\|t^{\alpha_2}\,\rho\,F\|_{L^{R_2}_T(L^{p})}\right)\,,
\end{equation}
where $F$ is defined in \eqref{d:F}. At this point, we bound the term $\|t^{\alpha_2}\,\rho(t)\,F(t)\|_{L^{R_2}_T(L^{p})}$ by following the computations of Paragraph \ref{sss:ub-w}:
first of all, \eqref{eq:w1} is now replaced by the control
$$
\|t^{\alpha_2}(\Id-\Delta)^{-1}\nabla P\|_{L^{R_2}_T(L^{p})}\,\leq\, C\, T^{\alpha_2\,+\,1/R_2}\,\|\vrho\|_{L^\infty_T(L^{p})}\,\leq\,C_T\,\wtilde{\mc N}(T)\,.
$$
Next, we have, noting that our conditions on the exponents imply that $\alpha_2>\gamma_0+\gamma_1,$
\begin{align*}
&\left\|t^{\alpha_2}\left(w\cdot\na w\,+\,w\cdot \na v\, +\,v\cdot\na w \,+\,v\cdot\na v\right)\right\|_{L^{R_2}_T(L^{p})} \\
&\qquad\leq\,C\,T^{\alpha_2-(\g_0+\g_1)}\left(\left\|t^{\g_0}\,v\right\|_{L_T^{R_0}(L^{\infty})}\,+\,\left\|t^{\g_0}\,w\right\|_{L_T^{R_0}(L^{\infty})}\right)
\left(\left\|t^{\g_1}\,\nabla v\right\|_{L_T^{T_1}(L^{p})}\,+\,\left\|t^{\g_1}\,\nabla w\right\|_{L_T^{R_1}(L^{p})}\right) \\
&\qquad\leq\,C\,T^{\alpha_2-(\g_0+\g_1)}\,\left(1+T^{\g_0}\right)\,\left(1+T^{\g_1}\right)\,\wtilde{\mc N}^2(T)\,,
\end{align*}
and  estimate \eqref{eq:divPu} becomes
\begin{align*}
\left\|t^{\alpha_2}\,\nabla(\Id-\Delta)^{-1}\div\bigl(P\,u\bigr)\right\|_{L^{R_2}_T(L^{p})}\,&\leq\,C\,T^{\alpha_2-\g_0}\,
\|P\|_{L^{R_1}_T(L^{p})}\,\left(\left\|t^{\g_0}\,v\right\|_{L^{R_0}_T(L^{\infty})}+\left\|t^{\g_0}\,w\right\|_{L^{R_0}_T(L^{\infty})}\right) \\
&\leq\, C\,T^{\alpha_2-\g_0+1/R_1}\,\left(1+T^{\g_0}\right)\,\wtilde{\cN}^2(T)\,.
\end{align*}
Finally, we have
\begin{align*}
\left\|t^{\alpha_2}\,\nabla(\Id-\Delta^{-1})\div u\right\|_{L^{R_2}_T(L^{p})}\,&\leq\,C\,T^{\alpha_2-\g_1}\,
\left(\left\|t^{\g_1}\,\nabla w\right\|_{L_T^{R_2}(L^{p})}+\left\|t^{\g_1}\,P\right\|_{L_T^{R_2}(L^{p})}\right) \\
&\leq\,C\,T^{\alpha_2-\g_1+1/(2R_2)}\,\left(1\,+\,T^{1/(2R_2)+\g_1}\right)\, \wtilde{\cN}(T)\,,
\end{align*}
and, arguing in a pretty similar way, we also get
\begin{align*}
&\left\|t^{\alpha_2}\nabla(\Id-\Delta)^{-1}\Bigl(\bigl(g(\rho)-g(1)\bigr)\,\dive u\Bigr)\right\|_{L^{R_2}_T(L^{p})} \\
&\qquad\qquad\leq\,C\,T^{\alpha_2-\g_1}\,\|\vrho\|_{L^\infty_T(L^\infty)}\,
\left(\left\|t^{\g_1}\,\Delta(\Id-\Delta)^{-1}P(\rho)\right\|_{L^{R_2}_T(L^p)}\,+\,\left\|t^{\g_1}\,\div w\right\|_{L^{R_2}_T(L^p)}\right) \\
&\qquad\qquad\leq\,T^{\alpha_2-\g_1+1/(2R_2)}\,\left(1\,+\,T^{1/(2R_2)+\g_1}\right)\, \wtilde{\cN}^2(T)\,.
\end{align*}

Putting all these bounds together, we end up with 
$$
\left\|t^{\alpha_2}\,\rho\,F\right\|_{L^{R_2}_T(L^{p})}\,\leq\,C_T\left(\wtilde{\mc N}(T)\,+\,\wtilde{\mc N}^2(T)\right)\cdotp
$$
Therefore, we can insert the previous inequality into \eqref{est:new-mr_prelim}: the application of a standard bootstrap argument allows us to find a time $T_*>0$ such that,
for all $t\in[0,T_*]$, one has
$$
\wtilde{\mc N}(t)\,\leq\,C\left(\|\vrho_0\|_{L^{p}\cap L^\infty}\,+\,\left\|w_0\right\|_{\dot B^{2-2/r}_{p,r}}\right)\,,
$$
for a suitable positive constant $C,$ which  completes the proof of Proposition \ref{p:max-reg_new}.
\end{proof}

\appendix

\section{Harmonic Analysis estimates} \label{app:h-a}

This appendix is devoted to the proofs of  Lemma \ref{l:symb} and Proposition \ref{p:tang-d_X}.

\subsection{Proof of Lemma \ref{l:symb}} \label{app_ss:lemma}
It is based on the following  Bony's paraproduct decomposition (first introduced in \cite{Bony})
for  the (formal) product  of two tempered distributions $u$ and $v$: 
\begin{equation}\label{eq:bony}
u\,v\;=\;T_uv\,+\,T_vu\,+\,R(u,v)\,,
\end{equation}
where we have defined
$$T_uv\,:=\,\sum_jS_{j-1}u\,\Delta_j v\andf
R(u,v)\,:=\,\sum_j\wt\Delta_j u\,\Delta_{j}v\with \wt\Delta_j:=\sum_{|j'-j|\leq1}\Delta_{j'}\,.$$
The above operator $T$ is called  \emph{paraproduct} whereas $R$ is called  \emph{remainder}. 
We refer to Chapter 2 of \cite{B-C-D} for a presentation of  continuity properties of the previous operators in the class of Besov spaces.
For the time being, we limit ourselves to pointing out that the generic term $S_{j-1}u\,\Delta_j v$ of $T_uv$ is spectrally supported on dyadic annuli
with radius of size about  $2^j$,
while the generic term $\wt\Delta_j u\,\Delta_{j}v$ of $R(u,v)$ is supported on dyadic balls  of size about $2^j$.

\medbreak
One can now start the proof to Lemma \ref{l:symb}.
By using Bony's decomposition \eqref{eq:bony} and a commutator's process, we get, denoting $\wt X:=(\Id-S_0)X,$
\begin{multline} \label{S9eq2} 
\p_X\s(D)g\,=\,\s(D)\dive(Xg)\,+\,\bigl[T_{X^k}; \s(D)\p_k\bigr]g\,-\,\s(D)\p_kT_{g}X^k -\s(D)\p_kR(\wt X^k,g)\\\,+\,T_{\s(D)\p_kg}X^k\,+\,R(\wt X^k,\s(D)\p_kg)\,+\,
\bigl( R(S_0X^k,\s(D)\p_kg)\,-\,\s(D)\p_kR(S_0 X^k,g)\bigr)\cdotp 
\end{multline}

Bounding the first term relies on the fact that multiplier operators in $S^{-1}$ map 
$B^0_{p,\infty}$ to $B^1_{p,\infty}$ (see \cite[Prop. 2.78]{B-C-D}) and that $L^p$ is embedded in $B^0_{p,\infty}.$
We thus have 
$$\begin{aligned}
\|\s(D)\dive(Xg)\|_{B^{1}_{p,\infty}}&\,\leq\, C\,\|\dive(Xg)\|_{B^0_{p,\infty}}\\
&\,\leq\, C\,\|\dive(Xg)\|_{L^p}\,\leq\, C\,\bigl(\|\p_Xg\|_{L^p}\,+\,\|\nabla X\|_{L^p}\,\|g\|_{L^\infty}\bigr)\cdotp
\end{aligned}$$
Next, to handle the third term of \eqref{S9eq2}, we use the fact that, being in $S^0,$ 
the operator $\sigma(D)\partial_k$ maps $B^1_{p,\infty}$ to itself (again, see \cite[Prop. 2.78]{B-C-D}), 
that the paraproduct operator $T$ maps $L^\infty\times B^1_{p,\infty}$ to $B^1_{p,\infty}$ 
and \cite[Rem. 2.83]{B-C-D}, and that $L^p$  is embedded in $B^0_{p,\infty}.$ We eventually get
 $$\begin{aligned}
\|\s(D)\p_kT_{g}X^k\|_{B^1_{p,\infty}}&\leq C \|T_gX\|_{B^1_{p,\infty}}\\&\leq C\|g\|_{L^\infty}\|\nabla X\|_{B^0_{p,\infty}}
\leq C\|g\|_{L^\infty}\|\nabla X\|_{L^p}.\end{aligned}
$$

Similarly, since  the remainder  operator $R$ maps $L^\infty\times B^1_{p,\infty}$ to $B^1_{p,\infty}$ and because, owing to the low frequency cut-off, 
we have
\begin{equation}\label{eq:DX}\|\wt X\|_{B^1_{p,\infty}}\leq C\|\nabla X\|_{B^0_{p,\infty}}\leq C\|\nabla X\|_{L^p},\end{equation}  
we readily get
$$\left\|\s(D)\p_k R(\wt X^k,g)\right\|_{B^{1}_{p,\infty}}\,\leq\, C\,\|g\|_{L^\infty}\,\|\na X\|_{L^p}\,.$$
Regarding the term $R(\wt X^k,\s(D)\p_kg),$ we just have to use \eqref{eq:DX} and that $R$ maps also $B^0_{\infty,\infty}\times B^1_{p,\infty}$
 to $B^1_{p,\infty},$ to get
$$\|R(\wt X^k,\s(D)\p_kg)\|_{B^{1}_{p,\infty}}\,\leq\, C\,\|\s(D)\p_kg\|_{B^0_{\infty,\infty}}\,\|\na X\|_{L^p}\,.$$
Since $\s(D)\p_k$ maps $B^0_{\infty,\infty}$ to itself, and because $L^\infty\hookrightarrow B^0_{\infty,\infty},$
that term also satisfies the required inequality.
\smallbreak
The term $T_{\sigma(D)\partial_kg}X^k$ turns out to be the only one that cannot be bounded in $B^1_{p,\infty}$ under our assumptions.
In fact, for that term, we use that  the paraproduct maps $B^{s-1}_{\infty,\infty}\times B^1_{p,\infty}$ to $B^s_{p,\infty}$ (as $s-1<0$)
to write (still using  \cite[Rem. 2.83]{B-C-D}),
$$\|T_{\sigma(D)\partial_kg}X^k\|_{B^s_{p,\infty}}\leq C\|\sigma(D)\partial_kg\|_{B^{s-1}_{\infty,\infty}} \|\nabla X^k\|_{L^p}.$$
Because $\sigma(D)\partial_k$ maps $B^{s-1}_{\infty,\infty}$ to itself, and $L^\infty\hookrightarrow B^{s-1}_{\infty,\infty},$
we get 
$$\|T_{\sigma(D)\partial_kg}X^k\|_{B^s_{p,\infty}}\leq C\|g\|_{L^\infty}\|\nabla X\|_{L^p}.$$
To conclude the proof, it is only a matter of bounding suitably the two commutators terms in \eqref{S9eq2}. 
 First of all, notice that since the general term of the paraproduct is spectrally supported in dyadic annuli, one may 
 find a  smooth function $\psi$ supported in some annulus centered at the origin, and such that
 \beq\label{S9eq3}
\left[T_{X^k}; \s(D)\p_k\right]g\,=\,\sum_{j\in\Z} \left[S_{j-1}X^k, \,\psi(2^{-j}D) \s(D)\p_k\right]\D_jg\,.
\eeq
For each fixed $j\in\Z$ and $k\in\{1,\cdots,d\},$  let us define $h_j^k:= i\cF^{-1}(\xi_k \psi(2^{-j}\cdot) \s).$
Then we have, thanks to the definition of $h^k_j(D)$ and the mean value formula,
$$\begin{aligned}
 \left[S_{j-1}X^k, \,\psi(2^{-j}D) \s(D)\p_k\right]\D_jg(x)&\,=  \int_{\R^d} h^k_j(y) \bigl(S_{j-1}X^k(x)-S_{j-1}X^k(x-y)\bigr)\D_jg(x-y)\,dy\\
&\,= \, -\int_0^1\int_{\R^d} h^k_j(y)\: y\cdot\nabla S_{j-1}X^k(x-\tau y)\,\D_jg(x-y)\,dy\,d\tau\\
&\,= \,- \int_0^1\int_{\R^d} h^k_j\Bigl(\frac z\tau\Bigr)\Bigl(\frac{z}\tau\Bigr)\cdot\nabla S_{j-1}X^k(x-z)\,\D_jg\Bigl(x-\frac z{\tau}\Bigr)\frac{dz}{\tau^d}\,d\tau\cdotp \end{aligned}$$
From the last line and convolution inequalities, we get
$$\bigl\| \bigl[S_{j-1}X^k, \,\psi(2^{-j}D) \s(D)\p_k\bigr]\D_jg\bigr\|_{L^p}\leq  \| |\cdot|h_j^k\|_{L^1}\|\D_jg\|_{L^\infty}\|\nabla S_{j-1}X^k\|_{L^p},$$
which, admitting for a while that 
\begin{equation}\label{eq:hjk}
 \| |\cdot|h_j^k\|_{L^1} \leq C2^{-j}
 \end{equation}
and using the definition of the norm in $B^1_{p,\infty}$  implies that 
$$\bigl\|\bigl[T_{X^k}; \s(D)\p_k\bigr]g\bigr\|_{B^1_{p,\infty}} \leq C \|g\|_{L^\infty}\|\nabla X\|_{L^p}.$$
In order to prove \eqref{eq:hjk}, we use the fact that 
performing integration by parts,  
$$(1+|z|^2)^d (z h_j^k(z))=  (2\pi)^{-d} \int e^{iz\cdot\xi} (\Id-\Delta)^d\nabla\bigl(\xi_k \psi(2^{-j}\cdot)\s\bigr)(\xi)\,d\xi.$$
As integration may be restricted to those $\xi\in\R^d$ such that $|\xi|\sim 2^j$  and since $\sigma$ is in $S^{-1},$ routine computations lead to 
$$(1+|z|^2)^d |z h_j^k(z)|\leq C 2^{-j}\quad\hbox{for all }\ z\in\R^d,$$
whence Inequality \eqref{eq:hjk}. 
\medbreak
In order to bound the last term of \eqref{S9eq2}, we use the fact that, owing to the properties of the localization of the Littlewood-Paley 
decomposition, we have for some suitable smooth function $\psi$ with compact support and value $1$ on some neighborhood 
of the origin,
$$
R(S_0X^k,\s(D)\p_kg)\,-\,\s(D)\p_kR(S_0 X^k,g)= \sum_{j=-1}^0 [\D_jS_0 X^k,\sigma(D)\psi(D)\pa_k]\wt\D_j g.
$$
Then, arguing as above and setting $h^k:= \cF^{-1} (i\xi^k \psi \sigma),$ we find that 
$$
 [\D_jS_0 X^k,\sigma(D)\psi(D)\pa_k]\wt\D_j g(x)= 
   \int_0^1\int_{\R^d} h^k(y)\: y\cdot\nabla \D_{j-1}S_0X^k(x-\tau y)\,\wt\D_jg(x-y)\,dy\,d\tau.
   $$
   Hence convolution inequalities and the fact that the only nonzero terms above correspond to $j=0,1,$ lead us to 
   $$\begin{aligned}   \|R(S_0X^k,\s(D)\p_kg)\,-\,\s(D)\p_kR(S_0 X^k,g)\|_{L^p} &\leq   
   C2^{-j}\|\nabla \D_{j-1} S_0X^k\|_{L^p} \|\wt\D_jg\|_{L^\infty}\\&\leq   C2^{-j}\|\nabla X\|_{L^p} \|g\|_{L^\infty}.\end{aligned}$$
   This  completes the proof to  Lemma \ref{l:symb}. \qed


\subsection{Proof of Proposition \ref{p:tang-d_X}} \label{app_ss:prop}

For all $1\leq j\leq d$ and  $\eta>0$, let us introduce the  following \emph{modified Riesz transform}:
\begin{equation} \label{def:riesz}
\mc R^{(\eta)}_j\,:=\,\pa_j\bigl(\eta\Id\,-\,\Delta\bigr)^{-1/2}\,,
\end{equation}
so that  $\mc R^{(\eta)}_i\mc R^{(\eta)}_j\,=\,\pa_i\pa_j\bigl(\eta\Id-\Delta\bigr)^{-1}$.
\smallbreak
 Proposition \ref{p:tang-d_X} follows from  Proposition \ref{p:tang} and  
the following lemma involving the tangential regularity with respect to \emph{only one} vector field.
\begin{lem} \label{l:riesz}
Let $p\in\,]1,\infty[\,$ and take a vector-field $X\in\L^{\infty,p}$. Let $g\in L^\infty$  be such that $g\in\L^p_X$ and $\mc R^{(\eta)}_i\mc R^{(\eta)}_jg\in L^\infty$ for some $\eta>0.$
 There exists a constant $C>0$ such that
$$
\left\|\pa_X\mc R^{(\eta)}_i\mc R^{(\eta)}_jg\right\|_{L^p}\,\leq\,C\,\Bigl(\bigl(\|\mc R_i^{(\eta)} \mc R_j^{(\eta)} g\|_{L^\infty}+\|g\|_{L^\infty}\bigr)\|\nabla X\|_{L^p}+\,
\|\pa_Xg\|_{L^p}\Bigr)\cdotp
$$
\end{lem}

For proving that lemma,  a few reminders concerning the Hardy-Littlewood \emph{maximal function}
 $M[f]$  of a function   $f$ in $L^1_{\rm loc}(\R^d)$  are in order.
 Recall that it is defined by
$$M[f](x)\,:=\,\sup_{r>0}\frac{1}{|B(x,r)|}\int_{B(x,r)}|f(y)|\,dy\,,$$
where $B(x,r)$ denotes  the ball in $\R^d$ of center $x$ and radius $r$, and $|B(x,r)|$ its Lebesgue measure.
\smallbreak
The following  statement is classical (for the proof, see e.g. Chapter 1 of \cite{Stein}).
\begin{lem} \label{l:max}
The following properties hold true.
\begin{enumerate}[(a)]
\item For any $1<p\leq\infty,$   there exist constants $0<c<C$ such that for any function $g$ in
$L^p(\R^d),$ 
$$
c\,\|g\|_{L^p}\,\leq\,\|M[g]\|_{L^p}\,\leq\, C\,\|g\|_{L^p}\,.
$$
\item Let $p, q\in\, ]1,\infty[\,$ or $p=q=\infty$. Let $\{f_j\}_{j\in\Z}$ be a sequence of functions in $L^p(\R^d)$ such that
$(f_j)_{\ell^q(\Z)}\in L^p(\R^d)$. Then there holds
\beno
\Bigl\|\bigl(M[f_j]\bigr)_{\ell^q}\Bigr\|_{L^p}\,\leq\,C\,\Bigl\|\bigl(f_j\bigr)_{\ell^q}\Bigr\|_{L^p}\,.
\eeno
\item For any fixed $\Phi\in  L^1(\R^d)$ such that $\Psi(x)=\sup\limits_{|y|\geq |x|}|\Phi(y)|\in L^1(\R^d)$ with $\int_{\R^d}\Psi(x)dx=A$, and any function $g$,  
we have for all $x\in\R^d$,
$$
\sup_{t>0}\bigl|g*\Phi_t(x)\bigr|\,\leq\,C\,M[g](x)\with  \Phi_t(x)\,:=\,t^{-d}\,\Phi(x/t).$$
\end{enumerate}
\end{lem}

We shall also need the following definition.
\begin{defi} \label{d:sobolev}
Let $s\in\R$ and $p\in\,]1,\infty[\,$. The \emph{homogeneous Sobolev space} $\dot W^{s,p}$ is defined as the set of $u\in\mc S'_h$ such that
$$
\|u\|_{\dot W^{s,p}}\,:=\,\bigl\|(-\Delta)^{s/2}u\bigr\|_{L^p}\,<\,\infty\,.
$$
\end{defi}
The  spaces $L^{p}$ and $\dot W^{s,p}$ may  be characterized in terms of Littlewood-Paley decomposition  as they come up as special Triebel-Lizorkin spaces (see e.g. \cite{R-S}, Chapter 2),.
\begin{prop} \label{p:W-charact}
Let $s\in\R$ and $p\in\,]1,\infty[\,$. Then one has the following equivalence of norms:
$$
\|u\|_{\dot W^{s,p}}\,\sim\,\left\|\bigl\|\bigl(2^{js}\,\dot\Delta_ju\bigr)_{j}\bigr\|_{\ell^2(\Z)}\right\|_{L^p}\andf
\|u\|_{L^p}\,\sim\,\left\|\bigl\|\bigl(\Delta_ju\bigr)_{j}\bigr\|_{\ell^2(j\geq-1)}\right\|_{L^p}\,.
$$
\end{prop}

\begin{proof}[Proof of Lemma \ref{l:riesz}]
We start the proof by remarking that
$$
\p_X \mc R^{(\eta)}_i\mc R^{(\eta)}_j g\,=\,\dive\left(X\, \mc R^{(\eta)}_i \mc R^{(\eta)}_j g\right)\,-\, \mc R^{(\eta)}_i \mc R^{(\eta)}_j g\,\dive X\cdotp
$$
Since
$$\left\| \mc R^{(\eta)}_i \mc R^{(\eta)}_j g\dive X\right\|_{L^p}\,\leq\,C\,\|\nabla X\|_{L^p}\,\|\mc R^{(\eta)}_i \mc R^{(\eta)}_j g\|_{L^\infty}\,,$$
 we just need  to bound  the term $\dive (X\,\mc R^{(\eta)}_i \mc R^{(\eta)}_j g)$ in $L^p.$
\medbreak
To this end, we resort  again  to    Bony's decomposition \eqref{eq:bony} and get
\begin{multline}\label{Saeq2a}
\dive(X\, \mc R^{(\eta)}_i \mc R^{(\eta)}_jg)\,=\,\mc R^{(\eta)}_i \mc R^{(\eta)}_j\dive(X\,g)\,+\,\p_k\left[T_{X^k}; \mc R^{(\eta)}_i \mc R^{(\eta)}_j\right]g \\
+\,\p_kT_{\mc R^{(\eta)}_i\mc R^{(\eta)}_jg}X^k\,-\,\mc R^{(\eta)}_i\mc R^{(\eta)}_j\p_kT_{ g}X^k+\,\p_k R(X^k,\mc R^{(\eta)}_i \mc R^{(\eta)}_jg)\,-\,\mc R^{(\eta)}_i\mc R^{(\eta)}_j\p_kR(X^k,g)\,.
 \end{multline}

The first term in the right-hand side of the previous relation may be bounded by means of Corollary \ref{c:Delta-Id} and identity \eqref{eq:d_Yf}: 
\begin{align*}
\left\|\mc R^{(\eta)}_i \mc R^{(\eta)}_j\dive(X\,g)\right\|_{L^p}\,&\leq\,C\,\left\|\dive(X\,g)\right\|_{L^p}\,\leq\,C\,\left(\|\pa_Xg\|_{L^p}\,+\,\|\nabla X\|_{L^p}\,\|g\|_{L^\infty}\right)\,.
\end{align*}

Next, since we have  $$
\left\|S_{m-1}\mc R^{(\eta)}_i\mc R^{(\eta)}_jg\right\|_{L^\infty}\,\leq\, C\,  \|\mc R^{(\eta)}_i\mc R^{(\eta)}_j g\|_{L^\infty}\,,
$$ 
we get for all $\ell\geq-1,$ thanks to Lemma \ref{l:max}(c), and using the fact that $S_{m-1}=0$ for $m\leq0,$
$$
\begin{aligned}
\left|\bigl(\D_\ell\p_kT_{\mc R^{(\eta)}_i\mc  R^{(\eta)}_jg}X^k\bigr)(x)\right|\,\leq \,&C\,2^\ell\,\sum_{|m-\ell|\leq 4}{M}\!\left[S_{m-1}\mc R^{(\eta)}_i\mc R^{(\eta)}_jg\,\D_{m}X^k\right](x)\\
\leq\, &C\,\sum_{|m-\ell|\leq 4}\left\|S_{m-1}\mc R^{(\eta)}_i\mc R^{(\eta)}_jg\right\|_{L^\infty}\,{M}\!\left[2^{m}\,\D_{m}X^k\right](x) \\
\leq\,&C\,\left\|\mc R^{(\eta)}_i\mc R^{(\eta)}_j g\right\|_{L^\infty}\,\sum_{|m-\ell|\leq 4,\: m\geq1}{M}\!\left[2^{m}\,\D_{m}X^k\right](x)\,.
\end{aligned}
$$
As a result, thanks to Proposition \ref{p:W-charact} (where we take $s=0$) and point (b) of Lemma \ref{l:max}, we get
$$
\begin{aligned}
\left\|\p_kT_{\mc R^{(\eta)}_i\mc R^{(\eta)}_j g}X^k\right\|_{L^p}\,&\leq\, C\,\left\|\mc R^{(\eta)}_i\mc R^{(\eta)}_j g\right\|_{L^\infty}\,
\Biggl\|\biggl(\sum_{\ell\geq1}\Bigl(M\bigl[2^{\ell}\,\dot\D_{\ell}X^k\bigr]\Bigr)^2\biggr)^{\!\!1/2}\Biggr\|_{L^p}\\ &\leq\,C\,\left\|\mc R^{(\eta)}_i\mc R^{(\eta)}_j g\right\|_{L^\infty}\,\|\na X\|_{L^p}\,.\end{aligned}$$

Using the same strategy for handling $\p_kT_{g}X^k$, we also obtain
$$
\left\|\mc R^{(\eta)}_i\mc R^{(\eta)}_j\p_kT_{g}X^k\right\|_{L^p}\,\leq\, C\,\left\|\p_k T_{g}X^k\right\|_{L^p}\,\leq\, C\,\|g\|_{L^\infty}\,\|\na X\|_{L^p}\,.$$

For  the third term of the second line in \eqref{Saeq2a}, we remove the low frequencies of $X$ and 
consider the \emph{modified} remainder defined by 
$$\wt R(X^k, \mc R^{(\eta)}_i \mc R^{(\eta)}_jg)=\sum_{m\geq0} 
\D_{m}X^k\,\wt{\D}_{m}\mc R^{(\eta)}_i\mc R^{(\eta)}_jg.$$
Then we  write that for all $\ell\geq-1,$ taking advantage of Lemma \ref{l:max}, 
$$
\begin{aligned}
\left|\D_\ell \p_k \wt R(X^k, \mc R^{(\eta)}_i \mc R^{(\eta)}_jg)(x)\right|\,\leq\, &C\,2^\ell\,\sum_{m\geq \max(0,\ell-5)}{M}\!\left[\D_{m}X^k\,\wt{\D}_{m}\mc R^{(\eta)}_i\mc R^{(\eta)}_jg\right](x)\\
\leq\,&C\,\|\mc R^{(\eta)}_i\mc R^{(\eta)}_j g\|_{L^\infty}\,\sum_{m\geq \max(0,\ell-5)}2^{\ell-m}\,{M}\!\left[2^{m}\D_{m}X^k\right](x)\,,
\end{aligned}
$$
  Hence, from  Proposition \ref{p:W-charact} and Young inequality for convolutions, we infer that
$$ \begin{aligned}\left\|\p_k \wt R(X^k,\mc R^{(\eta)}_i\mc R^{(\eta)}_jg)\right\|_{L^p}\,\leq\,&C\,\|g\|_{L^\infty}\biggl(\sum_{\ell\geq -5}2^{-\ell}\biggr)
\left\|\bigl\|\left({M}\!\left[2^{j}\,\D_{j}X\right](x)\right)_{j\in\N}\bigr\|_{\ell^2}\right\|_{L^p}\\
\leq\,&C\,\|g\|_{L^\infty}\,\|\na X\|_{L^p}\,.\end{aligned}$$
Obviously, the same estimate holds true for the term $\mc R^{(\eta)}_i\mc R^{(\eta)}_j\p_k\wt R(X^k,g).$ 
\smallbreak
The
low frequencies terms that have been discarded have to be treated together, that is, we have to bound
$\p_k[\Delta_{-1}X^k, \mc R^{(\eta)}_i \mc R^{(\eta)}_j]\wt\Delta_{-1}g$ in $L^p.$
and this may be done in the same way as at the end of the proof of Lemma \ref{l:symb}. 
We end up with 
$$ \|\p_k[\Delta_{-1}X^k, \mc R^{(\eta)}_i \mc R^{(\eta)}_j]\wt\Delta_{-1}g\|_{L^p}\leq C\|\Delta_{-1}\nabla X\|_{L^p} \|\wt\Delta_{-1} g\|_{L^\infty}
\leq C\|\nabla X\|_{L^p}\|g\|_{L^\infty}. $$

\medbreak
Finally, it remains us to handle the commutator term on the right-hand side of \eqref{Saeq2a}. We start by 
decomposing  $\mc R^{(\eta)}_i\mc R^{(\eta)}_j$ into 
\begin{equation}\label{eq:dec}
\mc R^{(\eta)}_i\mc R^{(\eta)}_j=\mc R_i\mc R_j - \eta(\eta\Id-\Delta)^{-1}\mc R_i\mc R_j\,,
\end{equation}
where $\mc R_j$ stands for the classical Riesz transform (which corresponds to take $\eta=0$ in \eqref{def:riesz}).
Let us first  consider the  part of the commutator corresponding to $\mc R_i\mc R_j.$ We have
$$
\left(\left[T_{X^k};\mc R_i\mc R_j\right]g\right)(x)=\sum_{m\geq1}
[S_{m-1}X^k,\mc R_i\mc R_j]\D_mg .
$$
Let $\theta(\xi)=-\xi_i\xi_j|\xi|^{-2}\varphi(\xi),$ with  $\varphi$ being
the function used in the Littlewood-Paley decomposition.
Since the generic term $[S_{m-1}X^k,\mc R_i\mc R_j]\D_mg $ is supported on dyadic annuli of size $2^m$,
one can write
$$\begin{aligned}
&[S_{m-1}X^k,\mc R_i\mc R_j]\D_mg\,= \\
&\qquad=\,\sum_{|m-\ell|\leq 4}\Bigl(S_{m-1}X^k\,\D_m\mc R_i\mc R_j\D_\ell g\,-\,\mc R_i\mc R_j\D_\ell\bigl(S_{m-1}X^k\,\D_mg\bigr)\Bigr)\\
&\qquad=\,\sum_{|m-\ell|\leq 4}\Bigl(S_{m-1}X^k\,\bigl(S_{m+1}\theta(2^{-\ell}D)g\,-\,S_{m}\theta(2^{-\ell}D)g\bigr)\,-\,\theta(2^{-\ell}D)\bigl(S_{m-1}X^k\,(S_{m+1}g-S_{m}g)\bigr)\Bigr)\end{aligned}$$
Therefore, by applying Abel rearrangement techniques an using that $\Delta_j=\dot\Delta_j$ for $j\in\N$, we get
$$\left[T_{X^k};\mc R_i\mc R_j\right]g\,=\,-\,\sum_{m\geq2}\sum_{|m-\ell|\leq 4}\Bigl[\dot\D_{m-2}X, \th(2^{-\ell} D)\Bigr]S_mg.$$
The general term of the above series  is spectrally supported in dyadic annuli of size about $2^m.$
Therefore, there exists some universal integer $N_0$ so that for all $q\in\Z,$
\begin{equation}\label{eq:commax}
\dot\D_q\pa_k\left[T_{X^k};\mc R_i\mc R_j\right]g\,=\,\sum_{|m-q|\leq N_0}
\sum_{|\ell-m|\leq4} \dot\D_q\pa_k\Bigl[\th(2^{-\ell} D),\dot\Delta_{m-2}X\Bigr]S_mg.
\end{equation}
Now, the point $(c)$ of Lemma \ref{l:max} ensures that for all $x\in\R^d,$
\begin{equation}\label{eq:commax1} \Big|\dot\D_q\pa_k\Bigl[\th(2^{-\ell} D),\dot\Delta_{m-2}X\Bigr]S_mg(x)\Big|
 \leq C2^q M\Bigl[\bigl[\th(2^{-\ell} D),\dot\Delta_{m-2}X\bigr]S_mg\Bigr](x)\end{equation}
 and the mean value formula gives us, denoting $\check\theta:=\cF^{-1}\theta,$ 
 $$
 \Bigl(\bigl[\th(2^{-\ell}D),\dot\Delta_{m-2}X^k\bigr]S_mg\Bigr)(x)=-2^{d\ell}
 \int_0^1\!\!\int_{\R^d} \check\theta(2^\ell z)\, z\cdot\nabla\dot\D_{m-2}X^k(x-\tau z) S_mg(x-z)\,dz\,d\tau\,,
 $$
 whence, performing a change of variables and setting $\Psi(z):=z\check \theta(z),$
  $$
 \Bigl(\bigl[\th(2^{-\ell}D),\dot\Delta_{m-2}X^k\bigr]S_mg\Bigr)(x)=-2^{-\ell}
 \int_0^1\!\!\int_{\R^d} \Bigl(\frac{2^\ell}\tau\Bigr)^d\Psi\Bigl(\frac{2^\ell z}{\tau}\Bigr)\cdot\nabla\dot\Delta_{m-2}X^k(x-z)
 S_mg\Bigl(x-\frac z\tau\Bigr)dz\,d\tau\,.
 $$
 From that latter relation, we deduce that 
 $$\begin{aligned}
M\Bigl[\bigl[\th(2^{-\ell} D),\dot\Delta_{m-2}X\bigr]S_mg\Bigr](x)
&\leq C2^{-\ell} \|g\|_{L^\infty}  \int_0^1\!\!\int_{\R^d}  \Bigl(\frac{2^\ell}\tau\Bigr)^d
\Bigl|\Psi\Bigl(\frac{2^\ell z}{\tau}\Bigr)\Bigr| M\bigl[\nabla\dot\Delta_{m-2}X^k(x-z)\bigr]dz\,d\tau\\
&\leq C2^{-\ell}\|g\|_{L^\infty}\,M\bigl[\dot\D_{m-2}\na X\bigr](x)\,.
\end{aligned}
$$
We now plug that inequality in \eqref{eq:commax1} and \eqref{eq:commax}, then take the  norm in $\ell^2(\Z)$ 
with respect to $q$ and eventually compute the norm in $L^p(\R^d).$
We end up with  
$$\left\|\partial_k\dot\Delta_q\left[T_{X^k}; \mc R_i\mc R_j\right]g\right\|_{L^p(\ell^2(\Z))}\leq C\|M(\dot\Delta_q\nabla X)\|_{L^p(\ell^2(\Z))}$$
Therefore, by applying Proposition \ref{p:W-charact} with $s=0$ and the point $(b)$ of Lemma \ref{l:max}, we finally get
\begin{equation}\label{eq:com1}
\left\|\p_k\left[T_{X^k};\mc R_i\mc R_j\right]g\right\|_{L^p}\,\leq\,C\,\|g\|_{L^\infty}\,\left\|\nabla X\right\|_{L^p}\,.
\end{equation}
To complete the proof, we have to bound the commutator term 
corresponding to the last part of \eqref{eq:dec}. To do this, we use the fact that
$$
\eta\left[T_{X^k}; (\eta\Id-\Delta)^{-1}\mc R_i\mc R_j\right]g= \eta\left[T_{X^k};(\eta\Id-\Delta)^{-1}\right]\mc R_i\mc R_jg 
+ \eta(\eta\Id-\Delta)^{-1}\left[T_{X^k};\mc R_i\mc R_j\right]g.
$$
To handle the last term, we just have to use \eqref{eq:com1} and the fact that 
$\eta(\eta\Id-\Delta)^{-1}$ maps $L^p$ to itself (uniformly with respect to $\eta$). 
For the other term, we use that, by embedding,  
$$
\|\eta\pa_k \left[T_{X^k};(\eta\Id-\Delta)^{-1}\right]\mc R_i\mc R_jg\|_{L^p}
\leq C\eta\| \left[T_{X^k};(\eta\Id-\Delta)^{-1}\right]\mc R_i\mc R_jg\|_{B^1_{p,1}}.
$$
Then, using the fact that the multiplier $\eta(\eta\Id-\Delta)^{-1}$ is in $S^{-2}$
(uniformly with respect to $\eta\leq1$)  and arguing as for bounding \eqref{S9eq3}, one ends up with 
$$
\|\eta\pa_k \left[T_{X^k};(\eta\Id-\Delta)^{-1}\right]\mc R_i\mc R_jg\|_{L^p}
\leq C\|\nabla X\|_{L^p} \|\mc R_i\mc R_jg\|_{B^{-1}_{\infty,1}}
\leq C\|\nabla X\|_{L^p} \|\mc R_i\mc R_jg\|_{B^{0}_{\infty,\infty}}.
$$
Clearly, in the above computations, the low frequencies of $g$ are not involved.
Hence, we actually have, using that $\mc R_i$ maps $\dot B^0_{\infty,\infty}$ to itself
and that $\dot B^0_{\infty,\infty}\hookrightarrow L^\infty,$ 
 $$\|\eta\pa_k \left[T_{X^k};(\eta\Id-\Delta)^{-1}\right]\mc R_i\mc R_jg\|_{L^p}
\leq C\|\nabla X\|_{L^p}\|g\|_{L^\infty}.$$
Summing up all the above estimate concludes the  proof of the lemma.
\end{proof}

\begin{small}
\subsection*{Acknowledgements}

The first author has been partially supported by the project INFAMIE  (ANR-15-CE40-0011)
 operated by the French National Research Agency (ANR).
The second author has been partially supported by the LABEX MILYON (ANR-10-LABX-0070) of Universit\'e de Lyon, within the program ``Investissement d'Avenir''
(ANR-11-IDEX-0007),  by the project BORDS (ANR-16-CE40-0027-01)
and  the programme ``Oberwolfach Leibniz Fellows'' by the Mathematisches Forschungsinstitut Oberwolfach in 2017. 
\end{small}


\end{document}